\documentclass[a4paper,12pt]{article}
\usepackage{amsmath}
\usepackage{amssymb}
\usepackage{setspace}
\usepackage{fullpage}
\usepackage{bbm}
\usepackage{amsthm}
\newtheorem{theorem}{Theorem}[section]
\newtheorem{lemma}[theorem]{Lemma}

\newtheorem{corollary}[theorem]{Corollary}

\theoremstyle{definition}

\newtheorem{definition}{Definition}
\newtheorem*{notation}{Notation}
\newtheorem{remark}{Remark}

\usepackage{pdfsync}
\def\PG{\mathrm{PG}}  
\def\V{\mathrm{V}}  

\def\Aut{\mathrm{Aut}}
\def\PGammaL{\mathrm{P}\Gamma\mathrm{L}}
\def\PGL{\mathrm{PGL}} 
\def\GL{\mathrm{GL}}

 \def\B{\mathcal{B}} 
\def\D{\mathcal{D}}  
 \def\K{\mathcal{K}}
\def\L{\mathcal{L}}  
 \def\P{\mathcal{P}}\def\Q{\mathcal{Q}}
\def\R{\mathcal{R}} \def\S{\mathcal{S}}
\def\FR{\mathcal{F}_{r,t,q}}
\def\K{\mathcal{K}}

\def\GammaL{\mathrm{\Gamma L}}

\def\F{\mathbb{F}}

\title{Field reduction and linear sets in finite geometry}
\author{ Michel Lavrauw \footnotemark[1] \and Geertrui Van de Voorde \thanks{This author is supported by the Fund for Scientific Research Flanders
(FWO -- Vlaanderen).}}
\begin{document}
\maketitle
\begin{abstract} 
Based on the simple and well understood concept of subfields in a finite field, the technique 
called `field reduction' has proved to be a very useful and powerful tool in finite geometry.
In this paper we elaborate on this technique. Field reduction for projective and polar spaces is
formalised and the links with Desarguesian spreads and linear sets are explained in detail. 
Recent results and some fundamental questions about linear sets and scattered spaces are studied.
The relevance of field reduction is illustrated by discussing applications
to blocking sets and semifields.

\end{abstract}


\section{Introduction}\label{S1}

In the last two decades a technique, commonly referred to as `field reduction', has been used in 
many constructions and characterisations in finite geometry. This is somehow surprising since the
technique is based on the well known structure of subfields of a finite field. In this paper we will
elaborate on this technique and explain how such a simple idea gives rise to highly non-trivial 
constructions and characterisations of geometric and algebraic structures.

For projective spaces the idea goes back to the 1960's, when B. Segre introduced Desarguesian spreads arising from field reduction \cite{segre}. At the end of the 1990's, the link with {\em linear sets} was introduced which renewed the interest for this technique, because it turned out to be very useful in the construction and characterisation of different kinds of objects in finite geometry. Field reduction for polar spaces was also introduced in the 1990's, in the study of $m$-systems \cite{shult}. 

\begin{notation}
An $n$-dimensional vector space over the finite field with $q$ elements $\F_q$ is denoted by $\V(n,q)$ or $\F_q^n$. The $(n-1)$-dimensional projective space corresponding to $\V(n,q)$ is denoted by $\PG(n-1,q)$ or 
$\PG(\F_q^n)$. A point in $\PG(n-1,q)$ corresponding a nonzero vector $v=(x_0,\ldots,x_{n-1})$ of $\V(n,q)$ is denoted by $\F_q(x_0,\ldots,x_{n-1})$ or $\F_q v$, reflecting the fact that every $\F_q$-multiple of $v$ defines the same projective point in $\PG(n-1,q)$.
If $U$ is a subspace of $\F_q^n$, then we denote the corresponding projective subspace of $\PG(n-1,q)$ by 
$\PG(U)$.
\end{notation}

The paper is organised as follows: in Section \ref{S2}, we formalise field reduction for projective spaces and explain the connection with Desarguesian spreads and Segre varieties. In Section \ref{S3}, we explain the technique for classical polar spaces, embedded in projective spaces. In Section \ref{S4}, we turn our attention to linear sets and finally we discuss two topics in which linear sets and field reduction play an important role: blocking sets (Section \ref{S5}) and semifields (Section \ref{S6}).

\section{Field reduction for projective spaces}\label{S2}

The structure of subfields of a finite field is well understood and 
it is well-known that we can consider the finite field with $q^t$ elements $\F_{q^t}$ as a $t$-dimensional vector space over $\F_q$. A point $\F_{q^t}v$ of $\PG(r-1,q^t)$ is a 1-dimensional subspace of $\F_{q^t}^r$ and
consists of the set of vectors
$S_v:=\{\alpha v ~:~\alpha \in \F_{q^t}\}$.
Now consider $\F_{q^t}^r$ as a vector space $V$ over $\F_q$. This means that $V$ is defined as the set of
vectors of $\F_{q^t}^r$, addition is as in $\F_{q^t}^r$, and so is scalar multiplication but the field of scalars is 
$\F_q$ instead of $\F_{q^t}$. Observe that $V$ has dimension $rt$, 
and clearly the set $S_v$ forms a $t$-dimensional subspace of $V$. 
More generally, let $\pi$ be a $(k-1)$-dimensional subspace of $\PG(r-1,q^t)$, with $k\in \{0,1,\ldots,r-1\}$. Suppose $\pi=\PG(U)$ with $U=\langle u_1,\ldots , u_k\rangle$. The set of vectors belonging to $U$ is
$$
S_U=\{\alpha_1 u_1 + \ldots + \alpha_k u_k~:~\alpha_1, \ldots, \alpha_k \in \F_{q^t}\}.
$$
Then clearly the set $S_U$ forms a subspace of $V$ of dimension $kt$. Summarizing we have the following.
\begin{lemma}
Each $(k-1)$-dimensional subspace $\pi=\PG(U)$ of $\PG(r-1,q^t)$
corresponds to a $(kt-1)$-dimensional subspace $\K(\pi)$ of $\PG(V)\cong \PG(rt-1,q)$
defined by the $kt$-dimensional subspace of $V$ spanned by the vectors of $S_U$.
\end{lemma}

This is the idea behind {\it field reduction}. We formalise this idea introducing the 
{\em field reduction map} $\FR$, defined as a map from the subspaces of $\PG(r-1,q^t)$ to
the subspaces of $\PG(rt-1,q)$:
\begin{eqnarray}\label{def:field_reduction_map}
\FR: \PG(r-1,q^t) \rightarrow \PG(rt-1,q)~:~\pi\mapsto \K(\pi),
\end{eqnarray}
where $\K(\pi)$ is as in the above Lemma.
We collect the properties of the field reduction map in the following lemma. The proof easily follows from the definitions but we include a proof to get used to the notation.

\begin{lemma}\label{lemma:field_reduction_map} Let $\mathcal P$ denote the set of points of $\PG(r-1,q^t)$, and consider $\FR$ as defined in
(\ref{def:field_reduction_map}).\\
(i) The field reduction map $\FR$ is injective.\\
(ii) If $\pi$ is a $(k-1)$-dimensional subspace of $\PG(r-1,q^t)$, then $\FR(\pi)$ has dimension $kt-1$, so
each subspace contained in the image of $\FR$ has dimension $kt-1$ for some $k\in\{0,1,\ldots,r-1\}$.\\ 
(iii) Any two distinct elements of $\FR(\mathcal{P})$ are disjoint.\\ 
(iv) Each point in $\PG(rt-1,q)$ is contained in an element of $\FR(\mathcal{P})$.\\
(v) $|\FR(\mathcal{P})|=(q^{rt}-1)/(q^t-1)$.\\
(vi) The intersection of elements in the image of $\FR$ also belongs to the image of $\FR$.
(vii) The span of elements in the image of $\FR$ is either the trivial subspace or can be written as the span of elements of $\FR(\mathcal{P})$.
\end{lemma}

\begin{proof}
(i) Suppose that $\FR(\pi_1)=\FR(\pi_2)$, with $\pi_1=\PG(U_1)$ and $\pi_2=\PG(U_2)$, then $S_{U_1}=S_{U_2}$, which implies that $U_1=U_2$ and $\pi_1=\pi_2$.

(ii) Every $S_U$ contains $q^{kt}$ vectors forming a vector space of dimension $kt$ over $\F_q$.

(iii) Suppose that $\FR(P_1)$ and $\FR(P_2)$, where $P_1=\F_{q^t}v$ and $P_2=\F_{q^t}w$ have a point in common, then $\alpha v=\beta w$, which implies that $S_{v}=S_w$, hence $\FR(P_1)=\FR(P_2)$.

(iv) Let $P$ be a point in $\PG(rt-1,q)$, say $P=\F_q w$, then $P$ belongs to $\FR(\F_{q^t}w)$.

(v) This follows from (i) and (iii).

(vi) The intersection of $\FR(\pi_1)$ and $\FR(\pi_2)$ is clearly equal to $\FR(\pi_1\cap \pi_2)$.

(vii) If $\pi_1=\langle P_1,\ldots,P_l\rangle$ and $\pi_2=\langle P_{l+1},\ldots,P_s\rangle$, then 
$$\langle \FR(\pi_1),\FR(\pi_2)\rangle=\langle \FR(P_1),\ldots,\FR(P_s)\rangle.$$ \end{proof}

\subsection{Desarguesian spreads}

A \index{(t-1)-spread}\index{Spread}{\em $(t-1)$-spread} in $\PG(n-1,q)$ is a set of $(t-1)$-spaces, partitioning the set of points in $\PG(n-1,q)$. Two spreads $\S_1$ and $\S_2$ in $\PG(n-1,q)$ are {\it equivalent} if there exists a collineation of $\PG(n-1,q)$ mapping one to the other.
The following theorem of Segre gives a necessary and sufficient condition for the existence of a $(t-1)$-spread in $\PG(n-1,q)$. We include a proof using the field reduction map.

\begin{theorem}{ \cite{segre}} There exists a $(t-1)$-spread in $\PG(n-1,q)$ if and only if $t$ divides $n$. \end{theorem}
\begin{proof}
If there exists a $(t-1)$-spread in $\PG(n-1,q)$, it is clear that the number of points in a $(t-1)$-space has to divide the number of points in $\PG(n-1,q)$. From this, it follows that $t$ has to divide $n$.
Conversely, suppose $n=rt$. Put 
\begin{eqnarray}\label{def:desarguesian_spread}
\D_{r,t,q}:=\FR(\mathcal{P})
\end{eqnarray}
where $\FR$ is defined as in (\ref{def:field_reduction_map}) and $\mathcal P$ denotes the set of points of
$\PG(r-1,q^t)$.
Then (ii), (iii) and (iv) of Lemma \ref{lemma:field_reduction_map} imply that $\D_{r,t,q}$ is a $(t-1)$-spread of $\PG(rt-1,q)$.
\end{proof}

A spread $\S$ in $\PG(n-1,q)$ is called {\it Desarguesian} if there exist natural numbers $r$ and $t$
such that $n=rt$ and $\S$ is equivalent to $\D_{r,t,q}$.

\begin{remark} By \cite{segre} a $(t-1)$-spread in $\PG(n-1,q)$, where $t$ is a divisor of $n$, can be also constructed as follows. Embed $\PG(rt-1, q)$ as a subgeometry of $\PG(rt-1, q^t)$ in the canonical way, i.e. by restricting the coordinates to $\F_q$. Let $\sigma$ be the automorphic collineation of $\PG(rt-1,q^t)$ induced by the field automorphism $x\rightarrow x^q$ of $\F_{q^t}$, i.e., $\sigma: \F_{q^t} (x_0,x_1,\ldots,x_{rt-1}) \mapsto \F_{q^t} (x_0^q,x_1^q,\ldots,x_{rt-1}^q)$. Then $\sigma$ fixes $\PG(rt-1,q)$ pointwise and one can prove that a subspace of $\PG(rt-1, q^t)$ of dimension $d$ is fixed by $\sigma$  if and only if it intersects the subgeometry $\PG(rt-1, q)$ in a subspace of dimension $d$ and that there exists an $(r - 1)$-space $\pi$ skew to the subgeometry $\PG(rt-1, q)$ (see \cite{28}). Let $P$ be a point of $\pi$ and let $L(P)$ denote the $(t-1)$-dimensional subspace generated by the conjugates of $P$, i.e., $L(P) = \langle P,P^\sigma,\ldots,P^{\sigma^{t-1}}\rangle$. Then $L(P)$ is fixed by $\sigma$ and hence it intersects $\PG(rt-1,q)$ in a $(t-1)$-dimensional subspace over $\F_q$. Repeating this for every point of $\pi$, one obtains a set $\mathcal{S}$ of $(t-1)$-spaces of the subgeometry $\PG(rt-1, q)$ forming a spread. This spread is equivalent to $\D_{r,t,q}$.
\end{remark}

A \index{Regulus}{\em regulus} in a projective space, or {\em $(t-1)$-regulus}  if we want to specify the dimension of the elements, is a set $\mathcal{R}$ of $q+1$ two by two disjoint $(t-1)$-spaces with the property that each line meeting three elements of $\mathcal{R}$ meets all elements of $\mathcal{R}$. If $S_1,S_2,S_3$ are mutually disjoint $(t-1)$-subspaces with $\dim\langle S_1,S_2,S_3\rangle=2t-1$, then there is a unique regulus $\R(S_1,S_2,S_3)$ containing $S_1,S_2,S_3$. A spread $\S$ is called \index{Regular spread}{\em regular} if the regulus $\R(S_1,S_2,S_3)$ is contained in $\S$ for each three different elements $S_1,S_2,S_3$ of $\S$.
We note that, if $q>2$, a $(t-1)$-spread of $\PG(2t-1,q)$ is Desarguesian if and only if it is regular \cite{Br}. 

Note that a Desarguesian spread satisfies the property that each subspace spanned by spread elements is partitioned by spread elements (Lemma \ref{lemma:field_reduction_map} (vii)). Spreads satisfying this property are called \index{Normal spread}{\em normal} or {\em geometric}. Clearly, a $(t-1)$-spread in $\PG(2t-1,q)$ is always normal. A $(t-1)$-spread $\S$ in $\PG(rt-1,q)$, with $r>2$, is normal if and only if $\S$ is Desarguesian \cite{Barlotti}. 
For a survey and self-contained proofs of these characterisations of Desarguesian spreads, we refer to \cite{Bader}.

To explain why the spread $\D_{r,t,q}$  is called  `Desarguesian', we need to consider the following incidence structure constructed from a spread.
Let $\S$ be a $(t-1)$-spread in $\PG(rt-1,q)$. Embed $\PG(rt-1,q)$ as a hyperplane $H$ in $\PG(rt,q)$. Consider the following incidence structure ${\mathcal A}(\S)=(\P,\mathcal{L},\mathrm{I})$, where $\mathrm{I}$ is symmetric containment:
\begin{itemize}
\item[$\mathcal{P}$:] points of $\PG(rt,q)\setminus H$;
\item[$\mathcal{L}$:] $t$-spaces of $\PG(rt,q)$ intersecting $H$ exactly in an element of $\S$.
\end{itemize}
Then the incidence structure ${\mathcal{A}}(\S)$ is a $2-(q^{rt},q^t,1)$-design with parallelism \cite{Barlotti}.
These are the same parameters as the parameters of the design obtained from points and lines of an
affine space ${\mathrm{AG}}(r,q^t)$.
If $r=2$, then the ${\mathcal{A}}(\S)$ is an affine translation plane of order $q^t$, and in this case this
construction is known as the \index{Andr\'e/Bruck-Bose construction} {\em Andr\'e/Bruck-Bose construction}.
The spread $\D_{r,t,q}$ obtained via field reduction is called Desarguesian because the incidence structure
${\mathcal{A}}(\D_{r,t,q})$ is isomorphic to the design obtained from the points and lines of an affine space
${\mathrm{AG}}(r,q^t)$. This means that for $r=2$, the projective completion of the affine 
plane ${\mathcal{A}}(\S)$ is a Desarguesian projective plane $\cong \PG(2,q^t)$ if and only if the
the spread $\S$ is a Desarguesian spread.

\bigskip

Since every linear transformation of $\V(r,q^t)$ can be considered as a linear transformation of $\V(rt,q)$, we have that $\GL(r,q^t)\leq \GL(rt,q)$ (see e.g. \cite[p. 139]{Kleidman}). 

The group of all semilinear transformations of the vector space $\V(r,q^t)$ is denoted by $\Gamma L(r,q^t)$. We show that $\GammaL(r,q^t)$ can be embedded in $\GammaL(rt,q)$. Any $\sigma \in \Aut(\F_{q^t})$ can be uniquely written as $\tau \circ \rho$, where $\tau$ is an element of $\Aut(\F_{q^t})$, fixing $\F_q$ pointwise and $\rho$ is an element of $\Aut(\F_q)$. Now $\tau$ induces is an $\F_q$-linear map of $\V(r,q^t)$, so, as seen before, $\tau$ can be naturally embedded into $\GL(rt,q)$.
Hence, if $A$ is an element of $\GL(r,q^t)$ (hence of $\GL(rt,q)$), then an element $\phi$ of $\GammaL(r,q^t)$ can be written as $A\circ \sigma=A \circ (\tau \circ \rho)=(A \circ \tau)\circ \rho \in \GammaL(rt,q)$. It is clear that two different elements of $\GammaL(r,q^t)$ correspond to different elements of $\GammaL(rt,q)$, so this procedure provides an embedding.

\subsection{The Segre variety}\label{segre}
In this section we explain the connection between subgeometries and the Segre variety using field reduction.
Let us first recall the difference between a subspace and a subgeometry. A $k$-dimensional {\em subspace} $U$ of $\PG(n,q)$, also called a {\em $k$-space}, is isomorphic to a projective space $\PG(k,q)$.
A subgeometry $B$ on the other hand is isomorphic to a projective space $\PG(k,q_0)$ 
for some subfield $\F_{q_0}$ of $\F_q$. We define a {\em subgeometry} $B$ by the set of points of a projective space $\PG(k,q)$ whose coordinates with respect to some fixed frame take values from a subfield 
$\F_{q_0}$ of $\F_q$.
In this case the subspaces of $B$ correspond to the
intersections of subspaces of $\PG(n,q)$ with $B$. We also say that $B$ is a subgeometry {\em over $\F_{q_0}$} or
{\em of order $q_0$}.
For instance, for $k=n$, we take in a projective space $\PG(n,q)$ the set of points $B$ that have coordinates in a subfield $\F_{q_0}$ of $\F_q$, together with all the intersections of subspaces of $\PG(n,q)$ with $B$. In this way we obtain a subgeometry over $\F_{q_0}$ ({\em canonical} with respect to the frame to which these coordinates are defined). This subgeometry is isomorphic to a projective space $\PG(n,q_0)$. If $q=q_0^2$, then 
$B$ is usually called a {\it Baer subgeometry}.

We have seen in the previous subsection that applying the field reduction map
$\FR$ to all points of a projective space yields a Desarguesian spread $\D_{r,t,q}$. 
If we apply the field reduction map $\FR$ to all points of a subgeometry $\PG(r-1,q)$ of $\PG(r-1,q^t)$, then we obtain a subset of $\D_{r,t,q}$ that forms one of the systems of a {\em Segre variety} $\mathcal{S}_{r-1,t-1}$. We will provide a proof here to give an explicit example of how field reduction works.

\begin{definition}
The {\em Segre map} $\sigma_{l,k}: \PG(l,q) \times \PG(k,q) \rightarrow \PG((l+1)(k+1)-1,q)$ is defined by
\[\sigma_{l,k}(\F_q(x_0, \ldots, x_l), \F_q(y_0, \ldots , y_k)) := \F_q(x_0 y_0, \ldots, x_0 y_k, \ldots , x_l y_0, \ldots , x_l y_k).\]
The image of the Segre map $\sigma_{l,k}$ is called the {\em Segre variety} $\S_{l,k}$. 
\end{definition}

If we give the points of $\PG((l+1)(k+1)-1,q)$ coordinates in the form $$\F_q(x_{00},x_{01},\ldots,x_{0k};x_{10},\ldots,x_{1k};\ldots;x_{l0},\ldots,x_{lk}),$$
 then it is clear that the points of the Segre variety $\S_{l,k}$ are exactly the points that have coordinates such that the matrix $(x_{ij})$, $0\leq i\leq l$, $0 \leq j \leq k$, has rank $1$ (see also \cite[Theorem 25.5.7]{GGG}).

By fixing a point in $\PG(l,q)$ and varying the point of $\PG(k,q)$, we obtain a $k$-dimensional space on $\S_{l,k}$.
For every point of $\PG(l,q)$ such a space exists, and the set of these subspaces, which are clearly disjoint, is called a {\em system} ({\em of maximal subspaces}). Similarly, by fixing a point in $\PG(k,q)$, we obtain an $l$-dimensional space on $\S_{l,k}$ by varying the point of $\PG(l,q)$; the set of these subspaces is again called a system (of maximal subspaces). Subspaces of different systems intersect each other in exactly one point, while subspaces within the same system intersect each other trivially. Moreover, each subspace lying on the variety $\S_{l,k}$ is contained in an element of one of these two systems.

Let $P$ be a point of $\PG(r-1,q^t)$, say $P=\F_{q^t}v$, for some nonzero vector $v=(X_0,\ldots,X_{r-1})$, $X_i\in \F_{q^t}$, so $P$ corresponds to the vector line containing the vectors with coordinates 
$(\lambda_jX_0,\lambda_jX_1,\ldots,\lambda_jX_{r-1})$, where $X_i,i=0,\ldots,r-1$ 
are fixed elements  of $\F_{q^t}$ and $\lambda_j$, $j=0,\ldots,q^t-1$ ranges over $\F_{q^t}$.

Now we show that a subgeometry $\Sigma \cong\PG(k-1,q)$ of $\PG(r-1,q^t)$ corresponds to one of the systems of
a Segre variety $\S_{k-1,t-1}$ contained in the Segre variety $\S_{r-1,t-1}$. 

\begin{theorem}\label{segrevariety}
If $\P_\Sigma$ is the set of points of a subgeometry $\Sigma \cong \PG(k-1,q)$ of $\PG(r-1,q^t)$ of order $q$, then $\FR(\P_\Sigma)$ is projectively equivalent the system of $(t-1)$-spaces of a 
Segre variety $\S_{k-1,t-1}$ contained in the Segre variety $\S_{r-1,t-1}$.
\end{theorem}
\begin{proof}
We give a proof for $k=r$, the proof for $k<r$ is easily obtained by replacing $r-k$ coordinates by zero's.

Let $\omega$ be a primitive element of $\F_{q^t}$, and consider the $\F_q$-basis $B=\{1,\omega,\omega^2,\ldots,
\omega^{t-1}\}$ for $\F_{q^t}$. For every $\lambda_j$ in $\F_{q^t}$, the element $\lambda_j X_i$ can be expressed in a unique way in terms of this basis, say $\lambda_j X_i=\sum_{s}x_{is}^j \omega^s$.

This implies that $\FR(P)$ is the $(t-1)$-dimensional projective space corresponding to the vector space  $S_v$ that consists of all vectors $$v^j:=\F_q\left(x_{00}^j,\ldots,x_{0(t-1)}^j; x_{10}^j,\ldots,x_{1(t-1)}^j; \ldots;x_{(r-1)0}^j,\ldots,x_{(r-1)(t-1)}^j\right),$$ with $j$ in $\{0,\ldots,q^{t}-1\}$.
Assume, without loss of generality, that $\Sigma$ is canonically embedded in $\PG(r-1,q^t)$ with respect to some
fixed frame.
It follows from above that applying the field reduction map $\FR$ to the 
point $P$ of $\Sigma$ with coordinates $\F_{q^t}(X_0,\ldots,X_r)$, with $X_i\in \F_q$,
gives the $(t-1)$-space of $\PG(rt-1,q)$ spanned by the points 
$$\F_q(X_0,0,\ldots,0;X_1,0,\ldots,0;\ldots;X_{r-1},0,\ldots,0),$$ 
$$ \F_q(0,X_0,\ldots,0;0,X_1,\ldots,0;\ldots;0,X_{r-1},\ldots,0),$$
$$\F_q(0,\ldots0,X_0;0,\ldots,0,X_1;\ldots;0,\ldots,0,X_{r-1}).$$
Hence $\FR(P)$ contains 
the points with coordinates $$\F_q(\mu_0X_0,\mu_1X_0,\ldots,\mu_{t-1}X_0;\mu_0X_1,\ldots,\mu_{t-1}X_1;\ldots;\mu_0X_{r-1},\ldots,\mu_{t-1}X_{r-1}),$$ 
$\mu_0,\ldots, \mu_{t-1}\in \F_q$. Since the matrix $(x_{ij})$ with $x_{ij}=\mu_iX_j$, 
corresponding to these coordinates has rank $1$, the points of $\FR(P)$ lie on the 
Segre variety $\S_{r-1,t-1}$. 
\end{proof}



\begin{corollary} 
The system of $(t-1)$-spaces of a Segre variety $\S_{k-1,t-1}$ in $\PG(rt-1,q)$, $k\leq r$, is projectively equivalent to a subset of $\D_{r,t,q}$, whereas the system of $(r-1)$-spaces of a Segre variety $\S_{r-1,u-1}$ in $\PG(rt-1,q)$, $u\leq t$, is projectively equivalent to a subset of $\D_{t,r,q}$.
\end{corollary}

\section{Field reduction for classical polar spaces} \label{S3}
In this section we elaborate on the concept of field reduction for classical polar spaces; starting from a classical
polar space in $\PG(r-1,q^t)$ we want to obtain a classical polar space in $\PG(rt-1,q)$.
We will see that field reduction for classical polar spaces is somewhat more involved than field reduction for projective spaces. The reason is the extra freedom that arises from the choice of the form that is used to obtain a polar space in $\PG(rt-1,q)$; different forms can give different types of polar spaces in 
$\PG(rt-1,q)$.

Polar spaces are incidence structures that can be defined axiomatically, see \cite{Veldkamp}, but here we only need the so-called {\it classical polar spaces}, i.e. polar spaces that are embedded in a projective space equipped with an appropriate sesquilinear form. 
A celebrated result of Tits \cite{tits} shows that every finite polar space of rank at least $3$ is {\em classical}.

\subsection{Classical polar spaces}

Let $Q(X_0,\ldots,X_n)=\sum_{i,j=0, i\leq j}^{n} a_{ij} X_i X_j$ be a quadratic form over $\F_q$. A {\em quadric} $\mathcal{Q}$ in $\PG(n, q)$ is the set of points $\F_q v$ that satisfy $Q(v) = 0.$ Let $q$ be a square and let $H(X_0,\ldots,X_n)=\sum_{i,j=0}^{n} a_{ij} X_i X_j^{\sqrt{q}}$ with $a_{ij} = a_{ji}^{\sqrt{q}}$, be a Hermitian form over $\F_q$. A {\em Hermitian variety} in $\PG(n, q)$, denoted by $\mathcal{H}(n,q)$, is the set of points $\F_q v$ that satisfy $H(v)= 0$. A quadric or Hermitian variety of $\PG(n, q)$ is called \index{Singular quadric}{\em singular} if there exists a coordinate transformation which reduces the form to one in fewer variables, otherwise, the quadric or Hermitian variety is called \index{Non-singular quadric}{\em non-singular}. 

If $n$ is even, all non-singular quadrics in $\PG(n,q)$ are projectively equivalent to the quadric with equation $X_0^2+X_1X_2+\ldots+X_{n-1}X_n=0$. These quadrics are \index{Parabolic quadric}called {\em parabolic} and are denoted by \index{$\mathcal{Q}(n,q)$}
$\mathcal{Q}(n,q)$. 

If $n$ is odd, a non-singular quadric in $\PG(n,q)$ is either projectively equivalent to the quadric with equation $X_0X_1+\ldots+X_{n-1}X_n=0$ or to the quadric with equation $f(X_0,X_1)+X_2X_3+\ldots+X_{n-1}X_n=0$, where $f$ is an irreducible homogeneous quadratic form over $\F_q$. Quadrics of the first type are called \index{Hyperbolic quadric}{\em hyperbolic} and are denoted by $\mathcal{Q}^+(n,q)$, quadrics of the second type are called {\em elliptic} and are denoted by $\mathcal{Q}^-(n,q)$.

The incidence structure defined by a nonsingular quadratic or Hermitian variety, consisting of the subspaces that are contained in the variety all form polar spaces. We use the same notation for the polar space and the varieties. The polar spaces $\mathcal{Q}(n,q)$, $\mathcal{Q}^+(n,q)$, and $\mathcal{Q}^-(n,q)$ are called the {\it orthogonal polar spaces}, respectively of {\it parabolic}, {\it hyperbolic} and {\it elliptic} type; the polar space $\mathcal{H}(n,q)$ is called the {\it Hermitian} or {\it unitary polar space}.

Examples of quadrics and Hermitian varieties can be constructed using a {\em polarity}, which is a collineation of order two, of $\PG(n,q)$ onto its dual space. The image of a subspace $\pi$ under a polarity is denoted by $\pi^{\bot}$ and is called the {\em polar (space) of $\pi$}. If a subspace $\pi$ is contained in $\pi^{\bot}$, then $\pi$ is called {\em absolute}. A polarity is determined by a field automorphism $\sigma$ and a non-singular matrix $A$. There are four types of polarities $(\sigma,A)$ of $\PG(n,q)$, listed below.
\begin{itemize}
\item[(i)] If $\sigma=1$, $q$ odd, $A=A^T$, then the polarity $(\sigma,A)$ is called an \index{Orthogonal polarity}{\em orthogonal} polarity.
\item[(ii)] If $\sigma=1$, $A = -A^T$, and $a_{ii}= 0$ for all $i$, then every point is an absolute point, $n$ should be odd, and the polarity $(\sigma,A)$ is called a \index{Symplectic polarity}{\em symplectic} polarity.
\item[(iii)] If $\sigma= 1$, $q$ even, $A = A^T$ and $a_{ii}\neq 0$ for some $i$, then the polarity $(\sigma,A)$ is called a \index{Pseudo-polarity}{\em pseudo-polarity}. 
\item[(iv)] If $\sigma\neq1$, then $q$ is a square, $\sigma:x\mapsto x^{\sqrt{q}}$, $A=A^{T\sigma}$ and $(\sigma, A)$ is called a \index{Hermitian polarity}{\em Hermitian} or \index{Unitary polarity}{\em unitary} polarity. 
\end{itemize}

 If $q$ is odd, then the absolute points of an orthogonal polarity form a quadric  in $\PG(n, q)$. If $q$ is a square, then the absolute points of a Hermitian polarity form a Hermitian variety in $\PG(n, q)$. 

The points of $\PG(n, q)$, $n$ odd, $n\geq 3$, together with the absolute subspaces of a symplectic polarity of $\PG(n,q)$ form a \index{Symplectic polar space} {\em symplectic polar space}, denoted by \index{$\mathcal{W}(n,q)$}$\mathcal{W}(n,q)$.

Together the polar spaces $\mathcal{Q}(n,q)$, $\mathcal{Q}^+(n,q)$, $\mathcal{Q}^-(n,q)$, $\mathcal{H}(n,q)$ and 
$\mathcal{W}(n,q)$ are called the {\it classical polar spaces}. If $r$ is the maximum dimension of a subspace contained in a classical polar space $\P$, then $r+1$ is the {\em rank} of $\P$.

The classical polar spaces can also be introduced using the theory of sesquilinear forms on $\F_q^n$.
If $Q$ denotes the quadratic form defining one of the orthogonal polar spaces, then the associated bilinear form
$\beta_Q(x,y):=Q(x+y)-Q(x)-Q(y)$ is symmetric, and if $q$ is odd, the quadratic form can be obtained from the bilinear form. Similarly $\mathcal{H}(n,q)$ corresponds to a $\sigma$-sesquilinear form $\beta_{\mathcal H}$, where $x^\sigma=x^{\sqrt{q}}$ (called {\it unitary form}), and $\mathcal{W}(n,q)$ corresponds to an alternating bilinear form $\beta_{\mathcal W}$. Note that if $\beta$ is a form corresponding to one of the classical polar spaces, and $\pi \mapsto \pi^{\bot}$ is the associated polarity, then we have $\beta(x,y)=0$ if and only if the hyperplane $(\F_q x)^\perp$ contains the point $\F_q y$. We call a subspace $\pi$ {\em totally isotropic} with respect to the form $\beta$ if for all points $\F_q x$ and $\F_q y$ in $\pi$, $\beta(x,y)=0$. A symmetric bilinear form with $\beta(x,x)\neq 0$ for some $x$ is called a {\em pseudo-symplectic}.

Let $\P$ be one of the orthogonal polar spaces in $\PG(n,q)$ with associated quadratic form 
$Q$ and bilinear form $\beta_Q$. A {\it hyperbolic line} of $\P$ is a line containing two points
$\F_q x$ and $\F_q y$ with $Q(x)=Q(y)=0$ and $\beta_Q(x,y)=1$. 

The classification
of quadratic forms over finite fields then gives us the following.
\begin{itemize}
\item The polar space $\Q^+(2n+1,q)$ is the orthogonal sum of $n+1$ hyperbolic lines. 
\item The polar space
$\Q^-(2n+1,q)$ is the orthogonal sum of $n$ hyperbolic lines and an {\it elliptic line}, corresponding
to $f(X_0,X_1)$.
\item The polar space $\Q(2n,q)$ is the orthogonal sum of $n$ hyperbolic lines and a point $\F_q x_0$ with 
$Q(x_0)\neq 0$, and we define the {\it sign of a parabolic quadric} $\Q(2n,q)$ to be $+1$ if $Q(x_0)$ is a square 
in $\F_q$ and $-1$ otherwise. 
\end{itemize}

The classical polar spaces as described above correspond to the classical groups: the orthogonal groups $O^+(2n,q)$, $O^-(2n,q)$, and $O(2n+1,q)$, the unitary group $U(n,q)$, and the symplectic group $Sp(n,q)$. The correspondence between the forms, the polar spaces, and the groups is given in the following table. 

{\small
\begin{center}
  \begin{tabular}{ |c|c|c| }
    \hline
     Quadratic form   & Polar space & Associated group  \\ \hline
    \hline
    {\it hyperbolic} & $\Q^+(2n-1,q)$  & $O^+(2n,q)$ \\  \hline
    {\it elliptic} & $\Q^-(2n-1,q)$ & $O^-(2n,q)$\\ \hline
    {\it parabolic} & $\Q(2n,q)$ & $O(2n+1,q)$\\ \hline
    \hline
    Bilinear form & Polar space & Associated group  \\ \hline
    \hline
    {\it hermitian} & $\mathcal{H}(n-1,q)$&$U(n,q)$\\ \hline
    {\it alternating}& $\mathcal{W}(n-1,q)$ & $Sp(n,q)$\\ \hline
    \hline
  \end{tabular}
\end{center}
}

\subsection{Field reduction and forms}
In order to obtain a polar space in $\PG(rt-1,q)$ from a polar space in $\PG(r-1,q^t)$, we associate a
form on $\F_q^{rt}$ starting from a form on $\F_{q^t}^r$  using the trace map.
Let $Tr$ denote the {\em trace map} from $\F_{q^t}$ to $\F_q$, 
$$Tr =Tr_{ \F_{q^t} / \F_q}: \F_{q^t}\mapsto \F_q: x\rightarrow x+x^q +\ldots+x^{q^{t-1}}.$$ 
Let $f$ be a form on $\F_{q^t}^r$, and let $L_\alpha$ be the map $\F_{q^t} \mapsto \F_q:  x \rightarrow Tr(\alpha x)$ with $\alpha \in \F_{q^t}$. The map $L_\alpha f=L_\alpha \circ f$ is clearly a form on $\F_{q}^{rt}$. If $f$ and $L_\alpha f$ are non-degenerate, then starting from a polar space in $\PG(r-1,q^t)$ with corresponding quadratic, alternating or hermitian form on $\F_{q^t}^r$, by {\em field reduction}, we can obtain a polar space in $\PG(rt-1,q)$.

 In \cite{gill}, N. Gill determines the conditions on $f$ and $\alpha$ to ensure that $L_\alpha f$ is non-degenerate if $f$ is non-degenerate.
 
 \begin{theorem}\cite[Theorem A]{gill}
 Let $\beta$ be a reflexive $\sigma$-sesquilinear form on $V(r,q^t)$, $Q$ a quadratic form, and $L_\alpha ~:~\F_{q^t}\rightarrow \F_q~:~x\mapsto Tr(\alpha x)$. Then\\
(i) $L_\alpha \beta$ is non-degenerate if and only if $\beta$ is non-degenerate and $\alpha \neq 0$; \\
(ii) if $q$ is even and $r$ is odd, then $L_\alpha  Q$ is degenerate;\\
(iii) if $q$ is odd or $r$ is even, then $L_\alpha  Q$ is non-degenerate if and only if $Q$ is non-degenerate and $\alpha \neq 0$.
 \end{theorem}

\begin{lemma}\label{lem:image}
Let $L_\alpha ~:~\F_{q^t}\rightarrow \F_q~:~x\mapsto Tr(\alpha x)$, $\alpha \in \F_{q^t}^\ast$. Suppose that $L_\alpha\beta$ and $L_\alpha Q$ are non-degenerate.
The image under the field reduction map of an absolute subspace of a polar space in $\PG(r-1,q^t)$, with associated sesquilinear form $\beta$ or quadratic form $Q$,  is an absolute subspace in $\PG(rt-1,q)$ of the polar space with associated sesquilinear form $L_\alpha\beta$ or quadratic form $L_\alpha Q$.
\end{lemma}
\begin{proof}
Suppose $\pi$ is an absolute subspace of the polar space in $\PG(r-1,q^t)$ with associated sesquilinear form $\beta$. Then for each two points $\F_{q^t}x,\F_{q^t}y$ in  $\pi$ we have
$\beta(\lambda x, \mu y)=0$, $\forall \lambda, \mu \in \F_{q^t}$, and hence
$L_\alpha\beta(\lambda x, \mu y)=0$, $\forall \lambda, \mu \in \F_{q^t}$. This implies that
for each two points $\F_q u$ and $\F_q v$ in $\FR(\pi)$ we have $L_\alpha\beta(u,v)=0$. It follows
that $\FR(\pi)$ is absolute with respect to $L_\alpha\beta$. The proof is analogous using a quadratic form.
\end{proof}

\subsubsection{Quadratic form field reduction}
The orthogonal polar spaces are defined by a quadratic form, and the field reduction of these spaces is studied using that form. 

In \cite{gill} the author determines the possible polar spaces that can be obtained for each quadratic form.
The approach used in \cite{gill} is from a group theory perspective, so we will go through the list of possibilities, and give elementary proofs of the results using our terminology. We obtain slightly different conditions.

Field reduction does not change the type of the orthogonal polar spaces in odd dimensional projective space. For the orthogonal polar space in even dimensional projective space, i.e. of parabolic type, the situation is more complicated.

\begin{theorem}
Let $Q$ be a non-degenerate quadratic form on $\F_{q^t}^r$ corresponding to the polar space $\mathcal Q$, and let $L_\alpha~:~\F_{q^t}\rightarrow \F_q~:~x\mapsto Tr(\alpha x)$, $\alpha \in \F_{q^t}^\ast$. Suppose $L_\alpha Q$ is non-degenerate and let ${\mathcal{Q}}'$ denote the polar space defined by $L_\alpha Q$. Then the following holds:\\
(i) if $\mathcal Q$ is of hyperbolic type, then so is ${\mathcal{Q}}'$;\\
(ii) if $\mathcal Q$ is of elliptic type, then so is ${\mathcal{Q}}'$;\\
(iii) if $\mathcal Q$ is of parabolic type, choose $\gamma \in \F_{q^t}^*$ a square if $sign({\mathcal{Q}})=1$ 
and a non-square if $sign({\mathcal{Q}})=-1$, then $q$ is odd and
\begin{itemize}
\item $\mathcal{Q}'$ is of parabolic type if $t$ is odd;
\item if $t$ is even then ${\mathcal{Q}}'$ is of hyperbolic type if (a) $q^{t/2}\equiv 1$ mod 4, and $\alpha\gamma$ is a non-square in $\F_{q^t}$ or (b) $q^{t/2}\equiv 3$ mod 4, and $\alpha\gamma$ is a square in $\F_{q^t}$;
\item $\mathcal{Q}'$ is of elliptic type in the remaining cases.
\end{itemize}
\end{theorem}
\begin{proof}
(i) Since $\mathcal Q$ has rank $r/2$, the polar space ${\mathcal{Q}}'$ will have rank at least $rt/2$, by Lemma \ref{lem:image}. This implies that ${\mathcal{Q}}'$ is of hyperbolic type.

\bigskip

(ii) If $\mathcal Q$ is of elliptic type, then $r$ is even, say $r=2n$, and we know that up to the choice of a basis it is the orthogonal sum of a $Q^-(1,q^t)$ and $n-1$ hyperbolic lines. The additivity of $L_\alpha$ and part (i) imply that $L_\alpha Q$ is the orthogonal sum of the field reduced $Q^-(1,q^t)$ and $n-1$ copies of a $(2t-1)$-dimensional space of hyperbolic type. 

So we only need to consider $L_\alpha Q$ with $Q$ of elliptic type in $\PG(1,q^t)$. 
If $L_\alpha Q$ is of hyperbolic type, then w.l.o.g. we may assume that the $(t-1)$-space $\pi=\{\langle (y,g(y))\rangle_{\F_q}~:~y \in \F_{q^t}^*\}$, with $g(Y)=\sum_{j=0}^{t-1}g_jY^{q^j}$ some $\F_q$-linear map, is totally isotropic with respect to $L_\alpha Q$, where $Q(X_0,X_1)=aX_0^2+bX_0X_1+cX_1^2$ irreducible in $\F_{q^t}[X_0,X_1]$. This means that
\begin{eqnarray}\label{eqn:trace1}
Tr(\alpha (ay^2+byg(y)+cg(y)^2))=0~\mbox{for all}~y \in \F_{q^t}.
\end{eqnarray}
But $Tr(\alpha (aY^2+bYg(Y)+cg(Y)^2))$
reduces, modulo $Y^{q^t}-Y$, to a polynomial of degree $\leq 2q^{t-1}$ in $Y$ which is less than $q^t$ if $q>2$. So if $q>2$, this implies that $Q(X_0,X_1)$ is reducible, which is a contradiction, and we may conclude that 
there is no totally isotropic $(t-1)$-space. Therefore $L_\alpha Q$ and hence ${\mathcal{Q}}'$ is of elliptic type. If $q=2$ then using $Tr(\gamma y^2)=Tr((\gamma)^{1/2} y)$, the polynomial $Tr(\alpha (aY^2+bYg(Y)+cg(Y)^2))$
reduces, modulo $Y^{2^t}-Y$, to a polynomial of degree $\leq 2^{t-1}$ in $Y$
which less than $2^t$, and hence again equation (\ref{eqn:trace1}) implies that
$aY^2+bYg(Y)+cg(Y)^2$ is reducible. We may conclude that also for $q=2$ 
the polar space ${\mathcal{Q}}'$ is of elliptic type.

\bigskip

(iii) Suppose $\mathcal Q$ is of parabolic type. By the non-degeneracy hypothesis this implies that $q$ is odd. If $t$ is odd, then $rt$ is odd and $L_\alpha Q$ must be of parabolic type. Next we consider the case when $t$ is even.

W.l.o.g. assume that $\mathcal Q$ is the orthogonal sum of $(r-1)/2$ hyperbolic lines and the point $\F_{q^t}x_0$, where $Q(x_0)=\gamma\neq 0$. Again, as in part (ii), using the additivity of $L_\alpha$ and part (i) imply that we only need to consider $L_\alpha Q$ with $Q(X_0)=\gamma X_0^2$. Note that $L_\alpha Q$ is of hyperbolic type if and only if 
$\exists u\in \F_{q^t}^*$ such that $Tr(\alpha \gamma u^2)=0$, otherwise $L_\alpha Q$ is of elliptic type.

First suppose that $t=2$.
We have
$$Tr(\alpha \gamma u^2)=0~\Leftrightarrow~ \alpha \gamma u^2 (1+(\alpha \gamma u^2)^{q-1})=0
~\Leftrightarrow~ (\alpha \gamma u^2)^{q-1}=-1.$$
If $\omega$ is a primitive element of $\F_{q^2}$, then this is equivalent to
$$(\alpha \gamma u^2)^{q-1}=\omega^{(q^2-1)/2}~\Leftrightarrow ~u^2=\frac{\xi \omega^{(q+1)/2}}{\alpha \gamma},$$
for some $\xi \in \F_q^*$. We have shown that $L_\alpha Q$ is of hyperbolic type if and only if 
$\exists u\in \F_{q^2}^*$ such that $u^2=\frac{\xi \omega^{(q+1)/2}}{\alpha \gamma}$, for some $\xi \in \F_q^*$. Note that $\xi$ is a square in $\F_{q^2}$ and $\omega^{(q+1)/2}$ is a square in $\F_{q^2}$ if and only if $q=3$ mod $4$.
This gives us the following conditions: $L_\alpha Q$ is of hyperbolic type if and only if
\begin{itemize}
\item[(a'')]  $q\equiv 1$ mod 4 and $\alpha \gamma$ is a non-square in $\F_{q^2}$;
\item[(b'')] $q\equiv 3$ mod 4 and $\alpha \gamma$ is a square in $\F_{q^2}$.
\end{itemize}

Next suppose $t>2$ even. If $t'=t/2$ then $\F_q \subset \F_{q^{t'}} \subset \F_{q^t}$ and 
$$Tr=Tr_{ \F_{q^t} / \F_q}=Tr_{ \F_{q^{t'}} / \F_q} Tr_{ \F_{q^t} / \F_{q^{t'}}} .$$
Applying parts (i) and (ii) and the arguments used for the case $t=2$, the conditions (a'') and (b'') become
\begin{itemize}
\item[(a')] $q^{t/2}\equiv 1$ mod 4 and $\alpha \gamma$ is a non-square in $\F_{q^t}$;
\item[(b')] $q^{t/2}\equiv 3$ mod 4 and $\alpha \gamma$ is a square in $\F_{q^t}$. 
\end{itemize}
Using the fact that $\gamma$ is a square if and only if $sign({\mathcal{Q}})=1$ concludes the proof.
\end{proof}

\subsubsection{Bilinear form field reduction}

For the Hermitian and symplectic polar spaces, we need to use the sesquilinear form to study the possible polar spaces that are obtained after field reduction. The following theorem, from \cite{gill}, summarises the results, where {\em atypical} indicates that the bilinear form is not of one of the prescribed types.

\begin{theorem}\cite[Theorem C]{gill}
Let $\beta$ be a non-degerate $\sigma$-sesquilinear form $\beta:\F_{q^t}^r\times \F_{q^t}^r\rightarrow \F_{q^t}$, with corresponding polar space of hermitian or symplectic type, and $L_\alpha=Tr\circ \alpha$ with $0\neq \alpha \in \F_{q^t}$.
Then the type of $L_\alpha\beta$ is as follows.
\begin{center}
{\small
  \begin{tabular}{ |c|c|c|c| }
    \hline
    Type of $\beta$ & Type of $L_\alpha\beta$ & Conditions&Embedding \\ \hline
    \hline
    hermitian&hermitian&$t$ odd, $\sigma(\alpha)=\alpha$& $U(r,q^t)\leq U(rt,q)$\\  \hline
    hermitian&atypical&$t$ odd, $\sigma(\alpha)\neq \alpha$ & --\\ \hline
    hermitian & alternating & $t$ even, $q$ even, $\sigma(\alpha)=\alpha$ & $U(r,q^t)\leq Sp(rt,q)$\\ \hline
    hermitian & alternating &$t$ even, $q$ odd, $\sigma(\alpha)=-\alpha$ & $U(r,q^t)\leq Sp(rt,q)$\\ \hline
    hermitian & atypical& $t$ even, $\sigma(\alpha)\neq \pm \alpha$& --\\ \hline
    hermitian & hyperbolic& $t$ even, $q$ odd, $r$ even, $\sigma(\alpha)=\alpha$ & $U(r,q^t)\leq O^+(rt,q)$\\ \hline
    hermitian & elliptic& $t$ even, $q$ odd, $r$ odd, $\sigma(\alpha)=\alpha$ & $U(r,q^t)\leq O^-(rt,q)$\\ \hline
    alternating &alternating&-- & $Sp(r,q^t)\leq Sp(rt,q)$\\ \hline
        pseudo-symplectic & pseudo-symplectic & $q$ even& --\\
    \hline
  \end{tabular}
  }
\end{center}
\end{theorem}
The last column of the table
provides a list of possible embeddings in terms of the associated groups. 

\subsubsection{Conclusion}
We summarise the possibilities for field reduction of the classical polar spaces in the following table, where the polar space in $\PG(rt-1,q)$ is obtained from the polar space in $\PG(r-1,q^t)$ using the map
$L_\alpha:\F_{q^t}\rightarrow \F_q~:~x\mapsto Tr_{\F_{q^t}/\F_q}(\alpha x)$ with $\alpha \in \F_{q^t}^*$.

{\small
\begin{center}
  \begin{tabular}{ |c|c|c|c| }
    \hline
    Polar space & Polar space & Conditions & Conditions\\ 
    in $\PG(r-1,q^t)$ & in $\PG(rt-1,q)$ & $r$, $t$, $q$ & $\alpha\neq 0$ \\ \hline
    \hline
    $hyperbolic$ & $hyperbolic$ & $r$ $even$ & --\\
    $elliptic$ & $elliptic$ & $r$ $even$ & --\\
       $parabolic$ & $parabolic$ & $r$ $odd$, $t$ $odd$, $q$ $odd$  &  --\\
    $parabolic$ & $hyperbolic$ $or$ $elliptic$ & $r$ $odd$, $t$ $even$ &  $(\ast)$ \\
 \hline
    $hermitian$ & $hermitian$ &$t$ $odd$,  $q$ $square$  & $\sigma(\alpha)=\alpha$ \\
        $hermitian$ & $symplectic$ & $t$ $even$ & $\sigma(\alpha)=-\alpha$ \\
   $hermitian$ & $hyperbolic$ & $t$ $even$, $q$ $odd$, $r$ $even$ &  $\sigma(\alpha)=\alpha$ \\
    $hermitian$ & $elliptic$ & $t$ $even$, $q$ $odd$, $r$ $odd$ & $\sigma(\alpha)=\alpha$\\
        $symplectic$ & $symplectic$  & $r$ $even$&  --\\ 

    \hline
  \end{tabular}
\newline  $(\ast)$ hyperbolic if ($q^{t/2}=1 \mod 4$ and $\alpha \gamma \notin \square$) or ($q^{t/2}=3 \mod 4$ and $\alpha \gamma\in \square$);
elliptic in the remaining cases, where $\square$ denotes the set of non-zero squares in $\F_{q^t}$ and $\gamma \in \square$ if $sign({\mathcal{Q}})=1$ and $\gamma \in \F_{q^t}^\ast\setminus \square$ if $sign({\mathcal{Q}})=-1$.

  \end{center}
}

Field reduction for polar spaces (also called the `trace trick') was used already in 1994 by Shult and Thas \cite{shult} to construct $m$-systems of polar spaces. Later on, the theory of {\em intriguing sets} extended that of $m$-systems and Kelly \cite{kelly} used field reduction to construct new examples of intriguing sets of polar spaces.
\section{Linear sets in projective spaces}\label{S4}
Linear sets generalise the concept of subgeometries in a projective space. They have many applications in finite geometry; linear sets have been intensively used in recent years in order to classify, construct or characterise various geometric structures, e.g. blocking sets and semifields that will be discussed at the end of this paper. For a further discussion of these and other applications, we refer to the survey of O. Polverino \cite{olga}.

\subsection{Definition}
To obtain a linear set in a projective space, some kind of reverse field reduction is used. The field reduction map
takes as input a subspace of $\PG(r-1,q^t)$ and returns a subspace of $\PG(rt-1,q)$. Or in other words
from an $\F_{q^t}$-subspace we obtain an $\F_q$-subspace.
A linear set, on the other hand, is defined by an $\F_q$-subspace and returns, not a subspace, but a subset of
a projective $\F_{q^t}$-linear space, i.e. a subset of some $\PG(r-1,q^t)$. 

More precisely, let $V={\mathbb{F}}_{q^t}^r$.
A set $L$ of points in $\PG(V)$ is called an {\em $\F_q$-linear set} ({\it  of rank $k$}) if there exists a subset $U$ of $V$ that forms a ($k$-dimensional) $\F_q$-subspace of $V$, such that $L=\B(U)$, where 
$$\B(U):=\{\F_{q^t}u:~u \in U\setminus \{0\}\}.$$ 
Often the notation $L_U$ is used to indicate the underlying subspace.
Obviously, if we say that the subset $U$ forms an $\F_q$-subspace of $V$, then 
we mean a subspace of the $rt$-dimensional space
that is obtained by considering $V$ as vector space over $\F_q$.
But from now on, we identify the $\F_q$-vector subspace $U$ with the subset $U$. This allows us to consider the projective subspace $\PG(U)$ in $\PG(rt-1,q)$. We summarize the above in the following diagram
\begin{displaymath}
\begin{array}{ccccccc}
U & \subseteq & {\mathbb{F}}_{q^t}^r & \longleftrightarrow & \F_q^{rt} & \supseteq & U\\
\\
 & & \updownarrow &  & \updownarrow & & \updownarrow \\
\\
L_U=\B(U) & \subseteq &\PG(r-1,q^t) & \longleftrightarrow & \PG(rt-1,q) & \supseteq & \PG(U)\\
\end{array}
\end{displaymath}

Recall that the field reduction map $\FR$ gives us a one-to-one correspondence between the points of $\PG(r-1,q^t)$ and the elements of a Desarguesian spread $\D_{r,t,q}$. This gives us a more geometric perspective on the notion of a linear set; namely, an $\F_q$-linear set is a set $L$ of points of $\PG(r-1,q^t)$ for which there exists a subspace $\pi$ in
$\PG(rt-1,q)$ such that the points of $L$ correspond to the elements of $\D_{r,t,q}$ that have a non-empty intersection with $\pi$. If there is no confusion possible, we will often identify the elements of $\D_{r,t,q}$ with the points of $\PG(r-1,q^t)$, i.e. a point $P$ is identified with its image under $\FR$.
This allows us to view $\B(\pi)$ as a subset of $\D_{r,t,q}$. This is illustrated by the following diagram, where as 
before $\mathcal P$ denotes the set of points of $\PG(r-1,q^t)$.
\begin{displaymath}
\begin{array}{ccccccc}
 & &\PG(r-1,q^t) & \longleftrightarrow & \PG(rt-1,q) & \supseteq & \pi\\
\\
 & & \downarrow & & \downarrow & & \Downarrow \\
\\
L=\B(\pi) & \subseteq& {\mathcal{P}} & \overset{\FR}{ \longleftrightarrow} & \D & \supseteq & \B(\pi) \\
\end{array}
\end{displaymath}

If $P$ is a point of $\B(\pi)$ in $\PG(r-1,q^t)$, where $\pi$ is a subspace of $ \PG(rt-1,q)$, then we define the \index{Weight}{\em weight of $P$} as $wt(P):=\mbox{dim}(\FR(P) \cap \pi)+1$. This makes a point to have weight 1 if its corresponding spread element intersects $\pi$ in a point. It is clear that a point of an $\F_q$-linear set of rank $k$ in $\PG(r-1,q^t)$ can have weight at most $\min\{ k,t\}$.


\begin{theorem}
Let $S=\B(\pi)$ be a linear set of rank $k>0$ and denote by $x_i$ the number of points of weight $i$, 
with $m=\min \{k,t\}$, then the following relations hold:\\
(i) $|S|=x_1+x_2+\cdots+x_m$\\
(ii) $x_1+(q+1)x_2+\frac{q^3-1}{q-1}x_3+\cdots+\frac{q^m-1}{q-1}x_m=\frac{q^k-1}{q-1}$\\
(iii) $|S|\leq \frac{q^k-1}{q-1}$\\
(iv) $|S|\equiv 1 \mod q$.
\end{theorem}
\begin{proof}
For (ii), count the couples $\{ (P\in \pi,\B(P))\}$, the other items follow directly.
\end{proof}

If $\pi$ intersects the elements of $\mathcal D$ in at most a point, i.e. the size of $\B(\pi)$ is maximal, or equivalently every point of $\B(\pi)$ has weight one, then we say that $\pi$ is {\it scattered with respect to $\mathcal D$}; in this case ${\mathcal B}(\pi)$ is called a {\it scattered linear set}. The notion of scattered linear sets was introduced in
\cite{BlLa2000}, where the following bound on the rank of a scattered linear set was obtained.
\begin{theorem}\label{max}\cite[Theorem 4.3]{BlLa2000}
A scattered $\F_q$-linear set in $\PG(r-1,q^t)$ has rank $\leq rt/2$.
\end{theorem}
Scattered linear sets that meet this bound are called {\em maximum scattered}. Maximum linear sets are related to interesting geometric objects such as two-weight codes, two-intersection sets and strongly regular graphs (see \cite{Lavrauw2001}). The connection with pseudoreguli will be explained in Section \ref{pseudo}, and the for the connection with particular classes of semifields (see Section \ref{S6}) we refer to \cite{MaPoTr2007} and more recent \cite{LaMaPoTr2013}.

We have the following useful lemma for linear sets.

\begin{lemma} \label{Desar} Let $\D$ be the Desarguesian $(t-1)$-spread of $\PG(rt-1,q)$. Let $\B(\pi)$ be a linear set of rank $k+1$, where $\pi$ is a $k$-dimensional space. For every point $R$ in $\PG(rt-1,q)$, contained in an element of $\B(\pi)$, there is a $k$-dimensional space $\pi'$, through $R$, such that $\B(\pi)=\B(\pi')$.
\end{lemma}
\begin{proof} Since all Desarguesian spreads are equivalent, we may assume 
$\D=\D_{r,t,q}$, the image of the set of points of $\PG(r-1,q^t)$ under the field reduction map $\FR$. Let $\varphi_\omega$, for $\omega \neq 0$, be the collineation of $\PG(rt-1,q)$ mapping a point $\F_q x$ of $\PG(rt-1,q)$ to $\F_q\omega x$.
Then $\varphi_\omega$ fixes each element of $\D_{r,t,q}$ since $\mathbb{F}_{q^t} x=\mathbb{F}_{q^t} \omega x$. Moreover, the set 
$$\{\F_q \omega x~:~\omega \in \F_{q^t}\setminus \{0\}\}$$ consists of the $(q^t-1)/(q-1)$ different points of $\B(\F_qx)$.

Let $R$ be a point contained in an element $\FR(P)$ of $\B(\pi)$, and let $T$ be a point in $\pi\cap \FR(P)$. It follows from the previous part that $R=T^{\varphi_\omega}$ for some $\omega \in \F_{q^t}$.

If $\F_q z \in \pi$, then $(\F_q z)^{\varphi_\omega}=\F_q \omega z \in \B(\F_q z) \in \B(\pi)$, and hence
$\B(\pi^{\varphi_\omega})\subset \B(\pi)$.
Since $\varphi_\omega$ is a collineation $\B(\pi^{\varphi_\omega})= \B(\pi)$.
\end{proof}

From this lemma, we have for every point $R$ in $\PG(rt-1,q)$, contained in an element of $\B(\pi)$, where $\pi$ is $(k-1)$-dimensional, there is a $(k-1)$-dimensional space $\pi'$, through $R$, such that $\B(\pi)=\B(\pi')$. This raises an important question: how many different subspaces $\pi'$ of dimension $(k-1)$ are there through a fixed point $R$ such that $\B(\pi')=\B(\pi)$? If $\B(\pi)$ is a regulus, this means $\pi$ is a line, then it is clear that through every point of an element of $\B(\pi)$, there is exactly one line $\pi'$ such that $\B(\pi')=\B(\pi)$, because through every point of a regulus, there exists a unique transversal line to this regulus. In Theorem \ref{raar} we will see that the answer to this question is not always equal to one. Some cases are well understood, but in general, this question remains open.

\subsection{Linear sets and projections of subgeometries}
It is clear from the definition (or from the link with Segre varieties described in Section \ref{segre}) that a subgeometry is a linear set, but a linear set is not necessarily a subgeometry. However, the following theorem by Lunardon and Polverino shows that every linear set is a projection of a subgeometry. For the particular case of linear {\em blocking sets}, this was proven in \cite{polito2}, for the case of { scattered linear sets}, but not using this terminology, it was shown already in 1981 in \cite{Limbos}.

Let $\Sigma=\PG(k-1,q)$ be a subgeometry of $\Sigma^\ast=\PG(k-1,q^t)$ and suppose there exists an $(k-r-1)$-dimensional subspace $\Omega^\ast$ of $\Sigma^\ast$ disjoint from $\Sigma$. Let $\Omega=\PG(r-1,q^t)$ be an $(r-1)$-dimensional subspace of $\Sigma^\ast$ disjoint from $\Omega^\ast$. Let $p_{\Omega^\ast,\Omega}$ denote the \index{Projection}  projection map defined by $x\mapsto \langle \Omega^\ast,x\rangle \cap \Omega$ for each point $x \in \Sigma^\ast \setminus \Omega^\ast$. The point set $\Gamma=p_{\Omega^\ast,\Omega}(\Sigma)$, i.e., the image of $\Sigma$ under the projection map $p_{\Omega^\ast,\Omega}$ is simply called the {\em projection} of $\Sigma$ from $\Omega^\ast$ into $\Omega$.

\begin{theorem}{ \cite[Theorem 1 and 2]{LP}}\label{projectie} If $\Gamma$ is a projection of $\PG(k-1,q)$ into $\Omega=\PG(r-1,q^t)$ with $k\geq r$, then $\Gamma$ is an $\mathbb{F}_q$-linear set of rank $k$ and $\langle \Gamma \rangle=\Omega$. Conversely, if $L$ is an $\F_q$-linear set of $\Omega$ of rank $k$ and $\langle L \rangle=\Omega=\PG(r-1,q^t)$, then either $L$ is a canonical subgeometry of $\Omega$ or there are a $(k-r-1)$-dimensional subspace $\Omega^\ast$ of $\Sigma^\ast=\PG(k-1,q^t)$ disjoint from $\Omega$ and a canonical 
 subgeometry $\Sigma$ of $\Sigma^\ast$ disjoint from $\Omega^\ast$ such that $L=p_{\Omega^\ast,\Omega}(\Sigma)$.
\end{theorem}


\begin{corollary} The set  $\B(\pi)$ of elements of $\D_{r,t,q}$, where $\pi$ is a $(k-1)$-dimensional space in $\overline{\Omega}=\PG(rt-1,q)$ is the projection of one of the two systems of a Segre variety $\S_{k-1,t-1}$ from a $(kt-rt-1)$-dimensional space $\overline{\Omega^\ast}$ skew from $\S_{k-1,t-1}$ and $\overline{\Omega}$ and vice versa.
\end{corollary}
\begin{proof}
Apply field reduction to the spaces $\Omega^\ast, \Sigma^\ast$ and $\Sigma$ in Theorem \ref{projectie} and use Theorem \ref{segrevariety}.\end{proof}

\bigskip

In the previous corollary, we have seen that $\B(\pi)$ is a projection of a Segre variety (this projection is not necessarily injective). Projections of Segre varieties are studied by Zanella in \cite{corrado}, where he shows that every embedded {\em product space} is the injective projection of a Segre variety. In \cite{LSZ}, the authors investigate the embedding of the product space $\PG(n-1,q)\times \PG(n-1,q)$ in $\PG(2n-1,q)$ and show that $\B(W)$, where $W$ is a scattered subspace of rank $n$ is an embedding of the product space $\PG(n-1,q)\times \PG(n-1,q)$. This embedding is of course covered by two systems of $(n-1)$-dimensional subspaces. However, they prove that $\B(W)$ contains $n$ systems of $(n-1)$-dimensional subspaces, and hence for $n>2$, contrary to what one might expect, there exist systems of maximum subspaces which are not the image of maximum subspaces of the Segre variety.

\bigskip

\subsection{The equivalence of linear sets}
A very natural question for linear sets is that of {\em equivalence}. We say that two sets $S_1$ and $S_2$ of points in $\PG(n,q^t)$ are $\PGammaL$-equivalent (resp. $\PGL$-equivalent) if there is an element $\phi$ in $\PGammaL(n+1,q^t)$ (resp. $\PGL(n+1,q^t)$) such that $\phi(S_1)=S_2$. In the previous section, we have seen that a linear set can be seen as the projection of a subgeometry. Subgeometries of the same order (embedded in the same projective space) are always $\PGL$-equivalent, but the equivalence problem for projections of subgeometries turns out to be quite hard. The following theorem  shows how the equivalence of linear sets, obtained as the projection of a subgeometry, can be translated into the equivalence of the spaces we are projecting from. For the particular case of $\F_q$-linear sets of rank $n+1$ in $\PG(2,q^n)$ (which is the case of linear {blocking sets}) this was proven in \cite{bonoli}.

\begin{theorem}\label{thm:equivalence} \cite[Theorem 3]{lavrauw}
Let $S_i$ be the $\F_q$-linear set of rank $r$ in $\PG(n-1,q^t)$, defined as the projection of $\Sigma_i\cong \PG(r-1,q)$ in $\Sigma^\ast/\Omega_i^\ast$, where $\langle \Sigma_i\rangle=\Sigma^*\cong \PG(r-1,q^t)$, $i=1,2,$ and suppose that $S_i$ is not a linear set of rank $s$ with $s<r$.
The following statements are equivalent.
\begin{itemize}
\item[(i)] There exists an element $\alpha \in {\mathrm{P\Gamma L}}(n,q^t)$ such that $S_1^\alpha=S_2$.
\item[(ii)] There exists an element $\beta \in \Aut(\Sigma^*)$ such that $\Sigma_1^\beta = \Sigma_2$ and $(\Omega_1^\ast)^\beta=\Omega_2^\ast$.
\item[(iii)] For all subgeometries $\Sigma\cong \PG(r-1,q)$ in $\Sigma^\ast$, skew to $\Omega_1^\ast$ and $\Omega_2^\ast$, there exist elements $\delta, \varphi,\psi \in \Aut(\Sigma^\ast)$, such that $\Sigma^\delta =\Sigma$ and $(\Omega_1^\ast)^{ \varphi\delta}=(\Omega_2^\ast)^\psi$, $\Sigma_1^ \varphi=\Sigma$ and $\Sigma_2^\psi=\Sigma$.
\end{itemize}
\end{theorem}

In this way, instead of studying the equivalence of linear sets directly, one can study the stabiliser in $\PGammaL(r,q^t)$ of a subgeometry $\PG(r-1,q)$ in $\PG(r-1,q^t)$: orbits of this group on subspaces of $\PG(r-1,q^t)$ are in one to one correspondence with $\PGammaL$-equivalence classes of the linear sets obtained by projecting from these subspaces.

For the particular case of linear sets of rank $3$ in $\PG(1,q^t)$, this reduces to the study of the orbits on points outside $\pi$ of the stabiliser of a subplane $\pi\cong \PG(2,q)$ in $\PG(2,q^t)$. Let us denote a linear set of size $q^2+1$ in $\PG(1,q^t)$ by a {\em club}. A scattered linear set of rank $3$ is a linear set containing $q^2+q+1$ points. The equivalence problem for linear sets of rank $3$ is solved in the following theorem.

\begin{theorem}\cite[Theorem 5]{lavrauw} \label{inequivalent} \begin{itemize}
\item[(i)] All clubs in $\PG(1,q^3)$ and all scattered linear sets of rank $3$ in $\PG(1,q^3)$ are projectively equivalent. 
\item[(ii)] All scattered linear sets of rank $3$ in $\PG(1,q^4)$ are projectively equivalent.
\item[(iii)] All clubs and all scattered linear sets of rank $3$ in $\PG(1,2^5)$ are equivalent, but there exist projectively inequivalent clubs and projectively inequivalent scattered linear sets of rank $3$ in $\PG(1,2^5)$. 
\item[(iv)] In all other cases, there exist non-equivalent clubs and non-equivalent scattered linear sets of rank $3$.
\end{itemize}
\end{theorem}

One can ask whether it is possible to translate the equivalence problem for linear sets $\B(\pi)$ and $\B(\pi')$, where $\pi$ and $\pi'$ are subspaces of $\PG(nt-1,q)$ in terms of equivalence of the subspaces $\pi$ and $\pi'$. This problem is still unsolved; we will give an idea why the `naive' approach is unsuccesful.

Let $S_1=\B(\pi_1)$ and $S_2=\B(\pi_2)$ be two $\F_q$-linear sets in $\PG(n-1,q^t)$ and let $\phi$ be an element of $\PGammaL(n,q^t)$ mapping $S_1$ onto $S_2$. For all points $P$ of $\pi_1$ of weight $1$, it is natural to define $\bar{\phi}(P)$ as the unique point $P'$ of $\pi_2$ such that $\B(P')=\phi(\B(P))$. Unfortunately, it turns out that this mapping $\bar{\phi}$ cannot always be extended to a collineation of $\PG(nt-1,q)$, as follows from the following theorem. 

\begin{theorem} \cite{lavrauw}\label{raar} Let $\B(\pi)$ be a scattered linear set of rank $3$ in $\PG(1,q^3)$, $q>4$. Let $P$ be a point of $\pi$. Then there is exactly one plane $\pi'\neq \pi$ through $P$ such that $\B(\pi)=\B(\pi')$.
\end{theorem}

\begin{remark}
Note that the planes $\pi$ and $\pi'$ are contained in the hypersurface ${\mathcal{Q}}_{2,q}$,  which was studied in \cite{LSZ}. We refer to \cite{LSZ} for more on this hypersurface and interesting hypersurfaces associated to scattered linear sets in higher dimensions.
\end{remark}
Let $\B(\pi)$ be a scattered linear set of rank $3$ in $\PG(1,q^3)$, $q>4$ and let $P$ be a point of $\pi$. The mapping $\bar{\phi}$ corresponding to the identity element of $\PGammaL(2,q^3)$, mapping $\B(\pi)$ onto itself cannot map a line of $\pi$ through $P$ onto a line of the plane $\pi'$ through $P$ obtained in Theorem \ref{raar},  since this would imply that there are two transversal lines through $P$ to the same regulus. Hence, $\bar{\phi}$ cannot be extended to a collineation of $\PG(5,q)$.

 It can be shown that the points of a line of $\pi$ are mapped by $\bar{\phi}$ onto the points of a conic in $\pi'$; the $q^2+q+1$ conics obtained in this way form a {\em bundle of conics}.


\subsection{The intersection of linear sets}

As seen before, subgeometries provide examples of linear sets. The study of the {\em intersection} of two subgeometries started in 1980 when Bose, Freeman and Glynn determined the possibilities for the intersection of two Baer subplanes in $\PG(2,q)$ \cite{Bose2}. In 2003, Jagos, Kiss and P\'or settled the case of intersecting Baer subgeometries in $\PG(n,q)$ \cite{jagos}. The problem of the intersection of subgeometries was solved in general by Donati and Durante in 2008, \cite{DoDu2008} where they proved the following.
\begin{theorem}{ \cite[Theorem 1.3]{DoDu2008}} Let $G$ and $G'$ be two subgeometries of order $p^t$ and $p^{t'}$ respectively of $\PG(n,q)$, $q=p^h$, with $t\leq t'$ and let $m=\gcd(t,t')$. If $G\cap G'$ is non-empty, then $G\cap G'=G_1\cup \ldots \cup G_k$, with $k\leq \frac{q-1}{p^{t'}-1}$ and with $G_1, \ldots, G_k$ subgeometries of order $p^m$ of independent subspaces of $\PG(n,q)$. 
\end{theorem}
They also showed the converse:

\begin{theorem}{ \cite[Theorem 1.4]{DoDu2008}}\label{subgeom}
Let $t$ and $t'$ be two positive divisors of $h$ with $t|t'$. Let $k\leq\min\{n+1, \frac{q-1}{p^{t'}-1}\}$ and let 
$G_1, \ldots, G_k$ be subgeometries of order $p^t$ 
of independent subspaces of $\PG(n,q)$, $q=p^h$. Then there exist two subgeometries 
$G$ and $G'$
of order $p^t$ and $p^{t'}$, 
respectively, of $\PG(n,q)$ such 
that $G\cap G'=G_1\cup \ldots \cup G_k$.
\end{theorem}

The intersection of linear sets in general is considerably more difficult: in general, it is not the union of linear sets contained in independent subspaces and the intersection problem is far from being solved.

The intersection of an {\em $\F_q$-subline} (which can be seen as an $\F_q$-linear set of rank $2$ with $q+1$ points) and a club of $\PG(1,q^t)$ was first determined in \cite{F} by Fancsali and Sziklai. However, in this proof, the authors used that all clubs of $\PG(1,q^t)$ are projectively equivalent, which is in general not true (see Theorem \ref{inequivalent}); in \cite{clubs}, the authors provide a correct proof.


By the following theorem, the intersection problem for an $\F_q$-subline and a linear set is completely solved.

\begin{theorem}\cite[Theorem 8 and 9]{lavrauw} An $\mathbb{F}_q$-subline intersects an $\F_q$-linear set of rank $k$ of $\PG(1,q^h)$ in $0,1,\ldots,\min\{q+1,k\}$ or $q+1$ points and for every subline $L\cong\PG(1,q)$ of $\PG(1,q^h)$, there is a linear set $S$ of rank $k$, $k\leq h$ and $k\leq q+1$, intersecting $L$ in exactly $j$ points, for all $0\leq j\leq k$.
\end{theorem}

This theorem was later extended by Pepe where she determines an upper bound on the size of the intersection of an $\F_{q^s}$-subline and a linear set. Note that, opposed to the case where $s=1$, this theorem does not show that all possibilities occur.
\begin{theorem}\cite[Proposition 5]{pepe} An $\F_q$-linear set $L$ of $\PG(1,q^t)$ either contains a fixed subline $\PG(1,q^s)$, $s|t$, or it intersects it in at most $\frac{t}{s}(q^{s-1}+q^{s-2}+\cdots+1)$ points. 
\end{theorem}

The following theorem deals with the slightly more general case of the intersection of two linear sets of rank $3$ in $\PG(1,q^t)$. But as mentioned before, the general problem remains wide open.

\begin{theorem}\label{inter} \cite[Theorem 23 and Remark 24]{lavrauw} Two $\F_q$-linear sets of rank 3 in $\PG(1,q^h)$, $q>3$, intersect in at most $2q+2$ points if $q$ is odd, and in at most $2q+3$ points if $q$ is even. For general $q$, there are two linear sets of rank $3$ in $\PG(1,q^t)$ intersecting in exactly $2q+2$ points. 
\end{theorem}
\subsection{Scattered linear sets and pseudoreguli}\label{pseudo}

We focus on scattered $\F_q$-linear sets of rank $3r$ in $\PG(2r-1,q^3)$. By Theorem \ref{max}, these scattered linear sets are {\em maximum scattered}. In this subsection, we will describe the relationship between scattered linear sets and pseudoreguli. 

First, it is worth noticing that all maximum scattered linear sets in $\PG(2r-1,q^3)$ are $\PGammaL$-equivalent (this was shown for $r=2$ in \cite[Proposition 2.7]{MaPoTr2007} and for general $r$ in \cite[Theorem 4]{lavrauw2}), whereas in $\PG(2r-1,q^t)$, $t>4$, there exist inequivalent maximum scattered linear sets (see Theorem \ref{inequivalent2}).

Let $\L$ be a scattered $\F_q$-linear set of rank $3r$ in $\PG(2r-1,q^3)$, then it can be shown (see \cite[Lemma 5]{lavrauw2}) that a line of $\PG(2r-1,q^3)$ meets $\L$ in $0,1,q+1$ or $q^2+q+1$ points and every point of $\L$ lies on exactly one $(q^2+q+1)$-secant to $\L$. Two different $(q^2+q+1)$-secants to $\L$ are disjoint and there exist exactly two $(r-1)$-spaces, called transversal spaces, meeting each of the $(q^2+q+1)$-secants. In the spirit of the pseudoregulus defined by Freeman in \cite{Freeman1980}, and extending the definition in \cite{MaPoTr2007}, the {\em pseudoregulus} $\P$ associated with $\L$ is defined as the set $\P$ of $\frac{q^{3r}-1}{q^3-1}$ lines meeting $\L$ in $q^2+q+1$ points.

The following theorem gives a geometric characterisation of a regulus and pseudoregulus.

\begin{theorem} \cite[Theorem 24]{lavrauw2} Let $q>2$. Let $\tilde{\S}$ be the point set of a set $\mathcal{S}$ of $q^3+1$ mutually disjoint lines in $\PG(3,q^3)$ such that the subline defined by three collinear points of $\tilde{\S}$ is contained in $\tilde{\S}$, then $\S$ is a regulus or a pseudoregulus.
\end{theorem}

We have seen that there is a pseudoregulus associated to every maximum scattered linear set in $\PG(2r-1,q^3)$. A maximum scattered linear set in $\PG(2r-1,q^t)$ has rank $rt$, but if $t>3$, we can not in general associate a pseudoregulus to it. For this reason, it makes sense to define maximum scattered linear sets {\em of pseudoregulus type}. Let $L$ be a scattered $\F_q$-linear set of $\Lambda=\PG(2r-1,q^t)$ of rank $rt$, $r,t\geq 2$, we say that $L$ is of pseudoregulus type if 
\begin{itemize}
\item[(i)] there exists $m=\frac{q^{rt}-1}{q^t-1}$ pairwise disjoint lines of $\Lambda$, say $s_1,s_2,\ldots,s_m$ such that
$$\ |L\cap s_i|=q^{t-1}+q^{t-2}+\ldots+q+1,~ \ \forall i=1,\ldots,m;$$
\item[(ii)] there exist exactly two $(r-1)$-dimensional subspaces $T_1$ and $T_2$ of $\Lambda$ disjoint from $L$ such that $T_j\cap s_i\neq \emptyset$ for each $i=1,\ldots,m$ and $j=1,2$.
\end{itemize}

The following theorem shows that this family of linear sets is not empty by constructing a family of linear sets $L_{\rho,f}$ that are maximum scattered and of pseudoregulus type. 
\begin{theorem} \cite{italianen} Let $T_1=\PG(U_1,\F_{q^t})$ and $T_2=\PG(U_2,\F_{q^t})$ be two disjoint $(r-1)$-dimensional subspaces of $\Lambda=\PG(V,\F_{q^t})=\PG(2r-1,q^t)$ $(t>1)$ and let $\Phi_f$ be the semilinear collineation between $T_1$ and $T_2$, induced by the invertible semilinear map $f=U_1\rightarrow U_2$ having as companion automorphism an element $\sigma\in Aut(\F_{q^t})$ such that $Fix(\sigma)=\F_q$. Then, for each $\rho \in \F_{q^t}^\ast$, the set

$$L_{\rho,f}=\{\langle u+\rho f(u)\rangle_{q^t}:u \in U_1\setminus \{0\}\}$$
is an $\F_q$-linear set of $\Lambda$ of pseudoregulus type whose associated pseudoregulus is $\mathcal{P}_{L_{\rho,f}}=\{\langle P,P^{\Phi_f}\rangle_{q^t}: P \in T_1\}$, with transversal spaces $T_1$ and $T_2$.

\end{theorem}
The authors also count the number of non-equivalent linear sets in the families $L_{\rho,f}$. Here, $\phi(t)$ denotes the Euler $\phi$-function, i.e. $\phi(t)$ is the number of integers $s$ smaller than $t$ and relatively prime to $t$.

\begin{theorem} \label{inequivalent2}\cite{italianen} In the projective space $\Lambda=\PG(2r-1,q^t)$ $(r\geq 2, t\geq 3)$ there are $\phi(t)/2$ orbits of scattered $\F_q$-linear sets of $\Lambda$ of rank $rt$ of type $L_{\rho,f}$ under the action of the collineation group of $\Lambda$.

\end{theorem}

Linear sets of pseudoregulus type are also studied because of the connection between linear sets and {\em semifields}, which will be discussed in Section \ref{S6}.

\section{Blocking sets and field reduction}\label{S5}
A {\em blocking set} in $\PG(n,q)$ {\em with respect to $k$-spaces} is a set $B$ of points such that every $k$-dimensional space in $\PG(n,q)$ contains at least one point of $B$. If we are considering blocking sets with respect to hyperplanes, we simply say that $B$ is a {\em blocking set}. A {\em minimal} blocking set $B$ (w.r.t. $k$-spaces) is a blocking set such that no proper subset of $B$ is a blocking set (w.r.t. $k$-spaces). 
A {\em small} blocking set in $\PG(n,q)$ with respect to $k$-spaces is a blocking set of size smaller then $3(q^{n-k}+1)/2$. A blocking set $B$ in $\PG(n,q)$ with respect to $k$-spaces is of R\'edei-type if there is a hyperplane containing $|B|-q^{n-k}$ points.

Linear blocking sets with respect to $(k-1)$-spaces in $\PG(n-1,q^t)$ were introduced by Lunardon \cite{L1}: he argues that an $\F_q$-linear set of rank $nt-kt+1$ is a blocking set with respect to $(k-1)$-spaces. This can easily be seen: let $\B(\pi)$ be an $\F_q$-linear set in $\PG(n-1,q^t)$, where $\pi$ is $(nt-kt)$-dimensional, then every $(kt-1)$-dimensional subspace of $\PG(nt-1,q)$ meets $\pi$ non-trivially, hence, the $(kt-1)$-spaces that arise from applying field reduction to the points of a $(k-1)$-space of $\PG(n-1,q^t)$ meet $\pi$, so $\B(\pi)$ is a blocking set w.r.t $(k-1)$-spaces.

Polito and Polverino \cite{polito2} showed that one can construct {\em minimal} linear blocking sets in $\PG(2,p^t)$, $p$ prime, $t\geq 4$ that are not of R\'edei-type. This contradicted a widespread conjecture which stated that a small minimal blocking set in $\PG(2,q^t)$ would necessarily be of R\'edei-type.

Soon after it was proven that there are small minimal linear blocking sets that are not of R\'edei-type, people conjectured that all small minimal blocking sets should be linear sets. This conjecture was stated formally by Sziklai in 2008 \cite{sziklai}. 
Up to our knowledge, this is the complete list of cases in which the linearity conjecture for blocking sets in $\PG(n,p^t)$, $p$ prime w.r.t. $k$-spaces has been proven.

\begin{itemize}
\item $t=1$ (for $n=2$, see \cite{blok}; for $n>2$, $k=n-1$, see \cite{heim}; for $n>2$, $k\neq n-1$, see \cite{sz})
\item $t=2$ (for $n=2$, see \cite{TS:97}; for $n>2$, $k=n-1$, see \cite{Storme-Weiner}; for $n>2$, $k\neq n-1$, see \cite{weiner})
\item $t=3$ (for $n=2$, see \cite{pol}; for $n>2$, $k=n-1$, see \cite{Storme-Weiner}; for $n>2$, $k\neq n-1$, see \cite{LSV,nora})
\item $k=n-1$ and $B$ is of R\'edei-type (for $n=2$, see \cite{simeon,redei2}; for $n>2$, see \cite{redei})
\item $k=n-1$ and dim$\langle B\rangle=t-1$ (see \cite{SziklaiVdV})
\item $k=n-1$ and dim$\langle B \rangle=t$ (see \cite{sz}).
\end{itemize}

It is shown in \cite{ik2} that, loosely speaking, if the linearity conjecture holds in $\PG(2,p^t)$, then it also holds for blocking sets with respect to $k$-spaces in $\PG(n,p^t)$, provided that $p$ is large enough.

When looking at the construction of a linear blocking set $B$ in $\PG(n-1,q^t)$ with respect to $(k-1)$-spaces, we see that we take $B$ to be $\B(\pi)$, where $\pi$ is an $(nt-kt)$-space in $\PG(nt-1,q)$, which is a blocking set with respect to $(kt-1)$-spaces. It is clear that every point set $\B(B')$, where $B'$ is a blocking set with respect to $(kt-1)$-spaces in $\PG(nt-1,q)$ is a blocking set with respect to $(k-1)$-spaces in $\PG(n-1,q^t)$. However, the difficulty lies in distinguishing when the obtained blocking set is minimal. The following theorem provides us with one case in which the minimality of $\B(B')$ can be proven. Note that a semioval is a set $S$ of points such that every point of $S$ lies on a unique tangent line to $S$. 
\begin{theorem}\cite{ik}  Let $\Omega$ be an $(nt-kt-2)$-dimensional subspace of $\PG(nt-1,q)$, let $\bar{B}$ be a minimal blocking set that is not a semioval, contained in the plane $\Gamma$ which is skew from $\Omega$ and let $K$ be the cone with vertex $\Omega$ and base $\bar{B}$. Let $B=\B(K)$, then $B$ is a minimal blocking set with respect to $(k-1)$-spaces in $\PG(n-1,q^t)$.
\end{theorem}

If we take $\bar{B}$ in the previous theorem to be a line, then the constructed blocking set is a linear blocking set and we may conclude that a linear blocking set is indeed minimal. For blocking sets with respect to lines in $\PG(n-1,q^t)$ this was already shown in \cite{L2} and for $k\neq n-1$, we could deduce the minimality of a linear blocking set from \cite[Lemma 3.1]{sz}.

\section{Semifields and linear sets}\label{S6}

Finite semifields are a generalisation of finite fields (where associativity 
of multiplication is not assumed) and the study of linear sets and field reduction
has been shown very useful in this theory. 

A {\it finite semifield} $({\mathbb{S}},+,\circ)$ is an algebra of finite dimension 
over a finite field $\mathbb F$ with at least two elements,
and two binary operations $+$ and $\circ$, satisfying the following axioms. 
\begin{itemize}
\item[(S1)] $({\mathbb{S}},+)$ is a group with identity element $0$.
\item[(S2)] $x\circ(y+z) =x\circ y + x\circ z$ and $(x+y)\circ z = x\circ z + y
\circ z$, for all $x,y,z \in {\mathbb{S}}$.
\item[(S3)] $x\circ y =0$ implies $x=0$ or $y=0$.
\item[(S4)] $\exists 1 \in {\mathbb{S}}$ such that $1\circ x = x \circ 1 = x$,
for all $x \in {\mathbb{S}}$.
\end{itemize}
Without axiom (S4) we have the definition of a {\it pre-semifields}.

Semifields are usually studied up to isotopism, because of the one-to-one correspondence between the isotopism classes of semifields and the isomorphism classes of the associated projective planes (by a theorem of A. A. Albert).
An {\it isotopism} (or {\it isotopy}) between two (pre-)semifields $({\mathbb{S}},\circ)$ and $({\mathbb{S}}',\circ')$ is a triple $(F,G,H)$ of nonsingular linear maps from ${\mathbb{S}}$ to ${\mathbb{S}}'$ such that 
$$x^F \circ' y^G = (x\circ y)^H,$$
for all $x,y \in {\mathbb{S}}$. If such an isotopism exists, the (pre-)semifields ${\mathbb{S}}$ and ${\mathbb{S}}'$ are called {\it isotopic} and the isotopism class of a (pre-)semifield ${\mathbb{S}}$ is denoted by $[{\mathbb{S}}]$.

The nuclei of a semifield are associative substructures of a semifield, 
and they arise in a similar way as the (commutative) center of non-commutative algebraic structures. However, 
while the commutative center is uniquely defined for a non-commutative structure, there are four different  associative substructures to consider for non-associative structures. These are called the nucleus, the left nucleus, the middle nucleus, and the right nucleus and are defined as follows.

The subset 
$${\mathbb{N}}_l({\mathbb{S}}):=\{x~:~ x \in {\mathbb{S}} ~|~ x \circ (y\circ
z)=(x\circ y)\circ z, ~\forall y,z \in {\mathbb{S}}\},$$ 
is called the {\it left nucleus} of ${\mathbb{S}}$. Analogously, 
one defines the {\it middle nucleus} 
$${\mathbb{N}}_m({\mathbb{S}}):=\{y~:~ y \in {\mathbb{S}} ~|~ x \circ (y\circ
z)=(x\circ y)\circ z, ~\forall x,z \in {\mathbb{S}}\},$$
and the {\it right nucleus}
$${\mathbb{N}}_r({\mathbb{S}}):=\{z~:~ z \in {\mathbb{S}} ~|~ x \circ
(y\circ z)=(x\circ y)\circ z, ~\forall x,y \in {\mathbb{S}}\}.$$ 
The intersection of these three nuclei is called the
{\it nucleus} or {\it associative center} ${\mathbb{N}}({\mathbb{S}})$, while the intersection
of the associative center and the {\it commutative center} $C({\mathbb{S}})$ (defined in the usual way) is called the
{\it center} \index{center} of $\mathbb S$ and denoted by $Z({\mathbb{S}})$.
One easily verifies that all of these substructures are finite fields and ${\mathbb{S}}$ can be
seen as a (left or right) vectorspace over these substructures, e.g. as a left vector space
$V_l({\mathbb{S}})$ over its left nucleus. 
Right multiplication in $\mathbb S$ by an element $x$ is
denoted by $R_x$, i.e. $y^{R_x}=y\circ x$, which
is an endomorphism of $V_l({\mathbb{S}})$.

We can now explain the geometric approach to finite semifields, which has
been very fruitful in recent years. 

This approach naturally breaks up
the study of semifields into different cases depending on the parameters of the semifield. 
Here we only give the correspondence theorem in the general setting,
where no assumptions on the nuclei or other properties of the semifield are made.

Let ${\mathbb{S}}$ be an $n$-dimensional semifield over $\F_q$, and denote the
dimensions of ${\mathbb{S}}$ over its left nucleus by $l$.
We define the
following subspaces of $ {\mathbb{S}}\times  {\mathbb{S}}$. For each
$x \in \mathbb S$, consider the set of vectors $S_x:=\{(y,y^{R_x}):
y \in {\mathbb{S}}\}$, and put $S_\infty:=\{(0,y): y \in
{\mathbb{S}}\}$. Then ${\mathcal
S}:=\{S_x: x \in {\mathbb{S}}\} \cup \{S_\infty\}$ is a spread of
${\mathbb{S}}\times  {\mathbb{S}}$. The set of endomorphisms
${\mathcal{R}}:=\{R_x~:~ x \in {\mathbb{S}}\}\subset {\mathrm{End}}(V_l({\mathbb{S}}))$
is called the {\it semifield spread set} \index{semifield spread set}
corresponding to ${\mathbb{S}}$.
Note  that by (S2) the spread set ${\mathcal{R}}$ is closed under addition and, by (S3), the
non-zero elements of ${\mathcal{R}}$ are invertible.

This means that $n$-dimensional semifields over
${\mathbb F}_q$, can be investigated via the
$\F_q$-vector space $U\subset \F_q^{ln}$ of dimension $n$ induced by the $\F_q$-vector space
${\mathcal{R}}\subset {\mathrm{End}}(V_l({\mathbb{S}}))$. 
Projectively this corresponds to the study of the $\F_q$-linear
set $L({\mathbb{S}}):=B(U)$ of rank $n$ in $\PG(l^2-1,q^{n/l})=\PG(V_l({\mathbb{S}}))$.
This leads us to the general correspondence theorem, which allows us to use
the geometric properties of linear sets in relation to the Segre variety, to solve isotopism
problems for finite semifields.

\begin{theorem}[\cite{Lavrauw2011}]
Let ${\mathcal{S}}_{l,l}(q^{n/l})$ denote the Segre variety in $\PG(l^2-1,q^{n/l})$, 
and denote its $(l-2)$nd secant variety by $\Omega$.
Let ${\mathcal{G}}$ denote the stabiliser inside the collineation group ${\mathrm{P\Gamma L}}(l^2,q^{n/l})$ of the two families of maximal subspaces on ${\mathcal{S}}_{l,l}(q^{n/l})$, and let  $X$ denote the set of linear sets of rank $n$ disjoint from $\Omega$.
Then the isotopism classes of semifields of order $q^n$, $l$-dimensional over their left nucleus, 
are in one-to-one correspondence with the orbits of ${\mathcal{G}}$ on the set X.
\end{theorem}

More details on this approach, the treatment of different special cases and several other links with 
finite geometry can be found in \cite{MaPoTr2007}, \cite{semifields}, \cite{Lavrauw2013}, \cite{italianen}.
The recent paper \cite{LaMaPoTr2013} is a nice illustration of how the study of linear sets of pseudoregulus type
associated to certain semifields can be used to solve isotopism problems for these semifields.


\begin{thebibliography}{99}
\bibitem{Andre} { J. Andr\'e.} \"Uber nicht-Desarguessche Ebenen mit transitiver Translationsgruppe. {\em Math. Z.} {\bf 60}  (1954), 156--186.
\bibitem{Bader} L. Bader and G. Lunardon. Desarguesian spreads. {\em Ricerche mat.} {\bf 60 (1)} (2011), 15--37.
\bibitem{simeon} S. Ball. The number of directions determined by a function over a finite field. {\em J. Combin. Theory Ser. A} {\bf 104 (2)} (2003), 341--350.
\bibitem{Barlotti} A. Barlotti and J. Cofman. Finite Sperner spaces constructed from projective and affine spaces. {\em Abh. Math. Semin. Univ. Hamb.} {\bf 40} (1974), 231--241.
\bibitem{blok}  A. Blokhuis.  On the size of a blocking set in $\PG(2,p)$. {\em Combinatorica} {\bf 14 (1)} (1994), 111--114. 
\bibitem{redei2}  A. Blokhuis, S. Ball, A.E. Brouwer, L. Storme, and T. Sz\H{o}nyi. On the number of slopes of the graph of a function defined on a finite field.  {\em J. Combin. Theory Ser. A} {\bf 86 (1)}  (1999), 187--196.

\bibitem{BlLa2000} A. Blokhuis and M. Lavrauw. Scattered spaces with respect to a spread in $\PG(n,q)$. {\em Geom. Dedicata} {\bf 81 (1-3)} (2000), 231--243.

\bibitem{bonoli} G. Bonoli and O. Polverino. $\F_q$-linear blocking sets in $\PG(2,q^4)$. {\em Innov. Incidence Geom.} {\bf 2} (2005), 35--56.
\bibitem{Bose2} { R.C. Bose, J.W. Freeman, and D.G. Glynn}. On the intersection of two Baer subplanes in a finite projective plane. {\em Utilitas Math.} {\bf 17} (1980), 65--77.
\bibitem{Br} R.H. Bruck and R.C. Bose. The construction of translation planes from projective spaces.  {\em J. Algebra}  {\bf 1}  (1964), 85--102.
\bibitem{28} L. R. Casse and C. M. OÕKeefe. Indicator sets for $t$-spreads of $\PG((s + 1)(t + 1) - 1, q)$. {\em Boll. Un. Mat. Ital. B} {\bf 7 (4)} (1990), 13--33.
\bibitem{DoDu2008} { G. Donati and N. Durante}. On the intersection of two subgeometries of $\PG(n,q)$. {\em Des. Codes Cryptogr.} {\bf 46 (3)} (2008), 261--267.
\bibitem{F} { Sz.L. Fancsali and P. Sziklai.} About maximal partial 2-spreads in ${\rm PG}(3m-1,q)$.  {\em Innov. Incidence Geom.} {\bf 4}  (2006), 89--102.
\bibitem{clubs} { Sz.L. Fancsali and P. Sziklai}. Description of the clubs. {\em Ann. Univ. Sci. Budapest. E\"otv\"os Sect. Math.} {\bf 51} (2009), 141--146.
\bibitem{Freeman1980} J.W. Freeman. Reguli and pseudo-reguli in $\PG(3,q^2)$. {\em Geom. Dedicata} {\bf 9} (1980), 267--280.
\bibitem{gill} N. Gill. Polar spaces and embeddings of classical groups. {\em New Zealand J. Math.} {\bf 36} (2007), 175--184.

\bibitem{nora} N.V. Harrach, K. Metsch, T. Sz\H{o}nyi, and Zs. Weiner. Small point sets of $\PG(n,p^{3h})$ intersecting each line in 1 mod $p^h$ points. {\em J. Geom.} {\bf 98 (1--2)} (2010), 59--78.
\bibitem{heim} U. Heim. Proper blocking sets in projective spaces. {\em Discrete Math.} {\bf 174 (1--3)} (1997), 167--176. 
\bibitem{GGG} J.W.P. Hirschfeld and  J.A. Thas. \emph{General Galois Geometries}. Oxford University Press, Oxford, 1991.
\bibitem{jagos} { I. Jagos, G. Kiss, and A. P\'or.} On the intersection of Baer 
subgeometries of $\PG(n,q^2)$. {\em Acta Sci. Math. } {\bf 69 (1--2)} (2003), 419--429. 
\bibitem{kelly} S. Kelly. Constructions of intriguing sets of polar spaces from field reduction and derivation. {\em Des. Codes Cryptogr.} {\bf 43 (1)} (2007), 1--8.
\bibitem{Kleidman} P. Kleidman and M. Liebeck. {\em The Subgroup Structure of the Finite Classical Groups.}  Cambridge University Press, Cambridge, 1990. 
\bibitem{Lavrauw2001} M. Lavrauw. Scattered spaces with respect to spreads, and eggs in finite projective spaces. PhD Dissertation, Eindhoven University of Technology, Eindhoven, 2001.
\bibitem{Lavrauw2011}
M. Lavrauw. Finite semifields with a large nucleus and higher secant varieties to Segre varieties. Adv.\ Geom.\ \textbf{11} (2011), 399--410.
\bibitem{Lavrauw2013} 
M. Lavrauw. Finite semifields and nonsingular tensors. {\em Des. Codes Cryptogr.} {\bf 68}, 1-3 (2013), 205--227.
\bibitem{LaMaPoTr2013}
M. Lavrauw, G. Marino, O. Polverino and R. Trombetti. Solution to an isotopism question concerning rank 2 semifields. To appear in {\em Journal of Combinatorial Designs}.
\bibitem{semifields} M. Lavrauw and O. Polverino O. {\em Finite semifields and Galois geometry}. Chapter in: De Beule J., Storme L. (eds.)
Current Research Topics in Galois Geometry. NOVA Academic Publishers, 2011.
\bibitem{LSZ} M. Lavrauw, J. Sheekey and C. Zanella. On embeddings of minimum dimension of ${\mathrm{PG}}(n,q)\times {\mathrm{PG}}(n,q)$. To appear in {\em Des. Codes Cryptogr.}
\bibitem{LSV} M. Lavrauw, L. Storme and G. Van de Voorde. A proof of the linearity conjecture for $k$-blocking sets in $\PG(n,p^3)$, $p$ prime. {\em J. Combin. Theory, Ser. A} {\bf 118 (3)} (2011), 808--818.
\bibitem{lavrauw}M. Lavrauw and G. Van de Voorde. On linear sets on a projective line. {\em Des. Codes Cryptogr.} {\bf 56 (2-3)} (2010), 89--104.
\bibitem{lavrauw2} M. Lavrauw and G. Van de Voorde. Scattered linear sets and pseudoreguli. {\em Electronic J. Combin} {\bf 20 (1)} (2013), P15.

\bibitem{Limbos} M. Limbos. A characterisation of the embeddings of $\PG(m,q)$ into $\PG(n,q^r)$. {\em J. Geom.} {\bf 16 (1)} (1981), 50--55.
\bibitem{L1} { G. Lunardon}.  Normal spreads. \emph{Geom. Dedicata} {\bf 75 (3)} (1999), 245--261.
\bibitem{L2} G. Lunardon. Linear $k$-blocking sets. {\em Combinatorica} {\bf 21 (4)} (2001), 571--581.
\bibitem{italianen} G. Lunardon, G. Marino, O. Polverino and R. Trombetti. Maximum scattered linear sets of pseudoregulus type and the Segre Variety ${\cal S}_{n,n}$. To appear in {\em J. Algebraic Combin.}

\bibitem{LP} G. Lunardon and O. Polverino. Translation ovoids of orthogonal polar spaces. {\em Forum Math.} {\bf 16 (5)} (2004), 663--669. 
\bibitem{MaPoTr2007} G. Marino, O. Polverino, and R. Trombetti. On $\F_q$-linear sets of $\PG(3,q^3)$ and semifields. {\em J. Combin. Theory, Ser. A} {\bf 114 (5)} (2007), 769--788.

\bibitem{pepe} V. Pepe. On the algebraic variety $\mathcal{V}_{r,t}$. {\em Finite Fields Appl.} {\bf 17 (4)} (2011), 343--349.
\bibitem{polito2} { P. Polito and O. Polverino}. On small blocking sets. {\em  Combinatorica } {\bf 18 (1)}  (1998), 133--137.
\bibitem{pol} {O. Polverino}. Small blocking sets in $\PG(2,p^3)$. {\em Des. Codes Cryptogr.} {\bf 20 (3)} (2000), 319--324.

\bibitem{olga} O. Polverino.  Linear sets in finite projective spaces. {\em Discrete Math.} {\bf 310 (22)} (2010), 3096--3107. 
\bibitem{segre} B. Segre. Teoria di Galois, fibrazioni proiettive e geometrie non desarguesiane. {\em Ann. Mat. Pura Appl.} {\bf 64} (1964), 1--76.
\bibitem{shult} E.E. Shult and J. Thas. $m$-systems of polar spaces. {\em J. Combin. Theory Ser. A} {\bf 68 (1)} (1994), 184--204.
\bibitem{redei} {L. Storme and P. Sziklai}. Linear pointsets and R\'edei type $k$-blocking sets in $\PG(n, q)$. {\em J. Algebraic Combin.} {\bf 14 (3)} (2001), 221--228. 
\bibitem{Storme-Weiner}{ L. Storme and Zs. Weiner.} On $1$-blocking sets in ${\rm PG}(n,q)$, $n\geq 3$. {\em Des. Codes Cryptogr.} {\bf 21 (1--3)} (2000), 235--251. 
\bibitem{TS:97} {T. Sz\H onyi}. Blocking sets in desarguesian affine and
projective planes. \emph{Finite Fields Appl.} {\bf 3 (3)} (1997), 187--202.
\bibitem{sz} { T. Sz\H{o}nyi and Zs. Weiner}. Small blocking sets in higher dimensions.  {\em J. Combin. Theory, Ser. A}  {\bf 95 (1)}  (2001), 88--101.
\bibitem{sziklai} P. Sziklai. On small blocking sets and their linearity. {\em J. Combin. Theory, Ser. A} {\bf 115 (7)}  (2008), 1167--1182.
\bibitem{SziklaiVdV} P. Sziklai and G. Van de Voorde. A small minimal blocking set in $\PG(n,p^t)$, spanning a $(t-1)$-space, is linear. {\em Des. Codes Cryptogr.} {\bf 68 (1-3)} (2013), 25--32.
\bibitem{tits} J. Tits. Buildings of spherical type and finite BN-pairs. {\em Springer-Verlag, Berlin, Lecture Notes in Mathematics} {\bf 386}, 1974.
\bibitem{ik} G. Van de Voorde. Constructing minimal blocking sets using field reduction. Preprint.
\bibitem{ik2} G. Van de Voorde. On the linearity of higher-dimensional blocking sets. {\em  Electronic J. Combin.} {\bf  17(1)} (2010), Research Paper 174, 16 pp.
\bibitem{Veldkamp} F.D. Veldkamp. Polar geometry. {\em Indag. Math.} {\bf 21} (1959), 512--551.
\bibitem{weiner} Zs. Weiner. Small point sets of $\PG(n,q)$ intersecting each $k$-space in 1 modulo $\sqrt{q}$ points. {\em Innov. Incidence Geom.} {\bf 1} (2005), 171--180.
\bibitem{corrado} C. Zanella. Universal properties of the Corrado Segre embedding. {\em Bull. Belg. Math. Soc. Simon Stevin.} {\bf 3 (1)} (1996), 65--79.
\end{thebibliography}
\end{document}